\pdfoutput=1
\documentclass[11pt]{article}
\usepackage{latexsym,amssymb,graphicx,amscd,amsmath,amsthm}
\usepackage{amsmath,amssymb,epsfig}
\usepackage{epstopdf}

\newcommand{\be}{\begin{equation}}
\newcommand{\ee}{\end{equation}}

\newtheorem{thm}{Theorem}[section]
\newtheorem{lem}[thm]{Lemma}
\newtheorem{prop}[thm]{Proposition}
\newtheorem{cor}[thm]{Corollary}
\newtheorem{de}[thm]{Definition}
\newtheorem{rem}[thm]{Remark}

\DeclareMathOperator{\card}{card}

\newcommand{\ba}{\begin{eqnarray}}

\newcommand{\ea}{\end{eqnarray}}

\usepackage[left=2.5cm,top=2.5cm,right=2.5cm]{geometry}

\begin{document}

\title{A Gram Determinant for Lickorish's Bilinear Form}
\bigskip

\author{Xuanting Cai\thanks{The author is with Mathematics Department, Louisiana State University, Baton Rouge, Louisiana  70803(email: xcai1@math.lsu.edu).}}

\date{}

\thispagestyle{empty}
\maketitle

\begin{abstract}
We use the Jones-Wenzl idempotents to construct a basis of the Temperley-Lieb algebra $TL_n$. 
This allows a short calculation for a Gram determinant of Lickorish's bilinear form on the Temperley-Lieb algebra.
\end{abstract}

\hspace{.2 in} {{\bf Keywords}: {\em Skein Theory, Temperley-Lieb Algebra.}}

\bigskip

\thispagestyle{empty}

\maketitle

\setcounter{page}{1}
\section{Introduction}
\setcounter{equation}{0}\label{intro}
In \cite{W}, Witten proposed the existence of 3-manifold invariants.
A mathematically rigorous definition was given by Reshetikhin and Turaev \cite{RT} using quantum groups and Kirby calculus \cite{K}.
Later, Lickorish \cite{L1} provided an alternative proof by using a bilinear form on the Temperley-Lieb algebra $TL_n$.
An important property Lickorish needed was that this bilinear form defined over $\mathbb{Z}[A,A^{-1}]$ is 
degenerate at certain $4(n+1)$th roots of unity and nondegenerate at $4i$th roots of unity for $i<n+1$.
Ko and Smolinsky obtained this result by using a recursive formula for the determinants of specific minors of this form \cite{KS}.
They did not give a closed form for the determinant. 
This was first done by Di Francesco, Golinelli and Guitter \cite{FGG}.
Di Francesco later gave a simpler proof.
In this paper, we give a short derivation by using a skein-theoretic approach together with a combinatorial proposition from Di Francesco \cite{F}.
In order to do this, we construct a nice basis $\mathfrak{D}_n$ for $TL_n$.
In fact, there have been several bases of $TL_n$ studied before.
See \cite{FGG}, \cite{F}, or \cite{GS}.
It turns out that $\mathfrak{D}_n$ is a rescaled version of the basis used in \cite{F},
but the properties of Jones-Wenzl idempotents significantly simplify the calculation.
Our skein-theoretic approach is motivated by the colored graph basis for TQFT modules developed in Blanchet, Habegger, Masbaum, Vogel's paper \cite{BHMV}.
A skein theoretic derivation of a Gram determinant for the type B Temperley-Lieb algebra is given in \cite{CP}.

\section{Temperley-Lieb Algebra}
\setcounter{equation}{0}
\label{TLAlgebra}
Let $F$ be an oriented surface with a finite collection of points specified in its boundary $\partial F$.
A link diagram in the surface $F$ consists of finitely many arcs and closed curves in $F$,
with a finite number of transverse crossings, each assigned over or under information.
The endpoints of the arcs must be the specified points in $\partial F$. We define the skein of $F$ as follows:

\begin{de}
Suppose $A$ is a variable. Let $\Lambda$ be the ring $\mathbb{Z}[A,A^{-1}]$
localized by inverting the multiplicative set generated by elements of $\{A^n-1\mid n\in\mathbb{Z}^+\}$.
The linear skein $\mathcal{S}(F)$ is the module of formal linear sums over $\Lambda$ of link
diagrams in $F$ quotiented by the submodule generated by the skein relations:
	\begin{enumerate}
		\item $L \cup U= \delta L$, where $U$ is a trivial knot, $L$ is a link in $F$ and $\delta=(-A^{-2}-A^2)$;
		\item 
		$\begin{minipage}{0.2in}\includegraphics[width=0.2in]{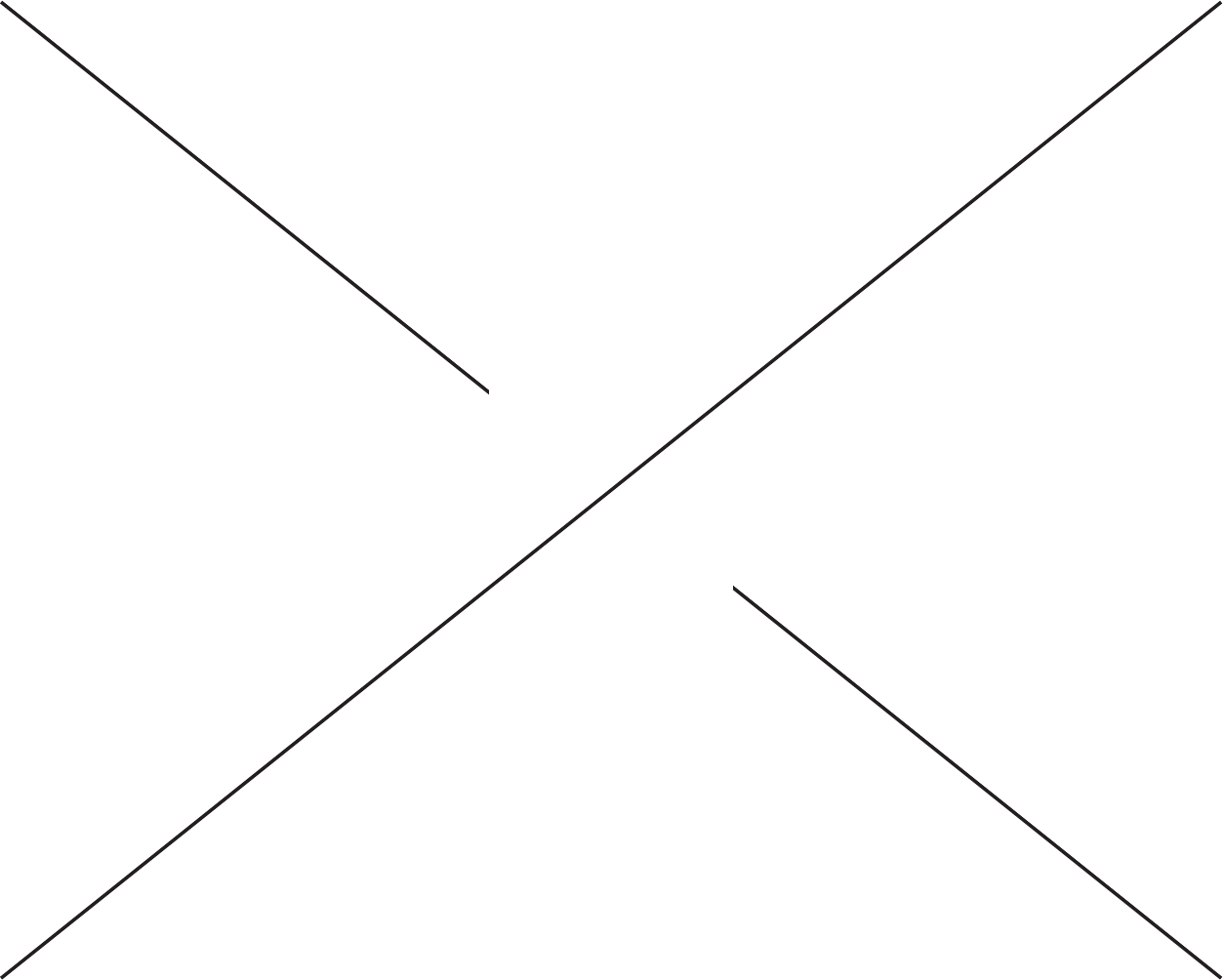}\end{minipage}=\ A^{-1}\ \begin{minipage}{0.2in}\includegraphics[width=0.2in]{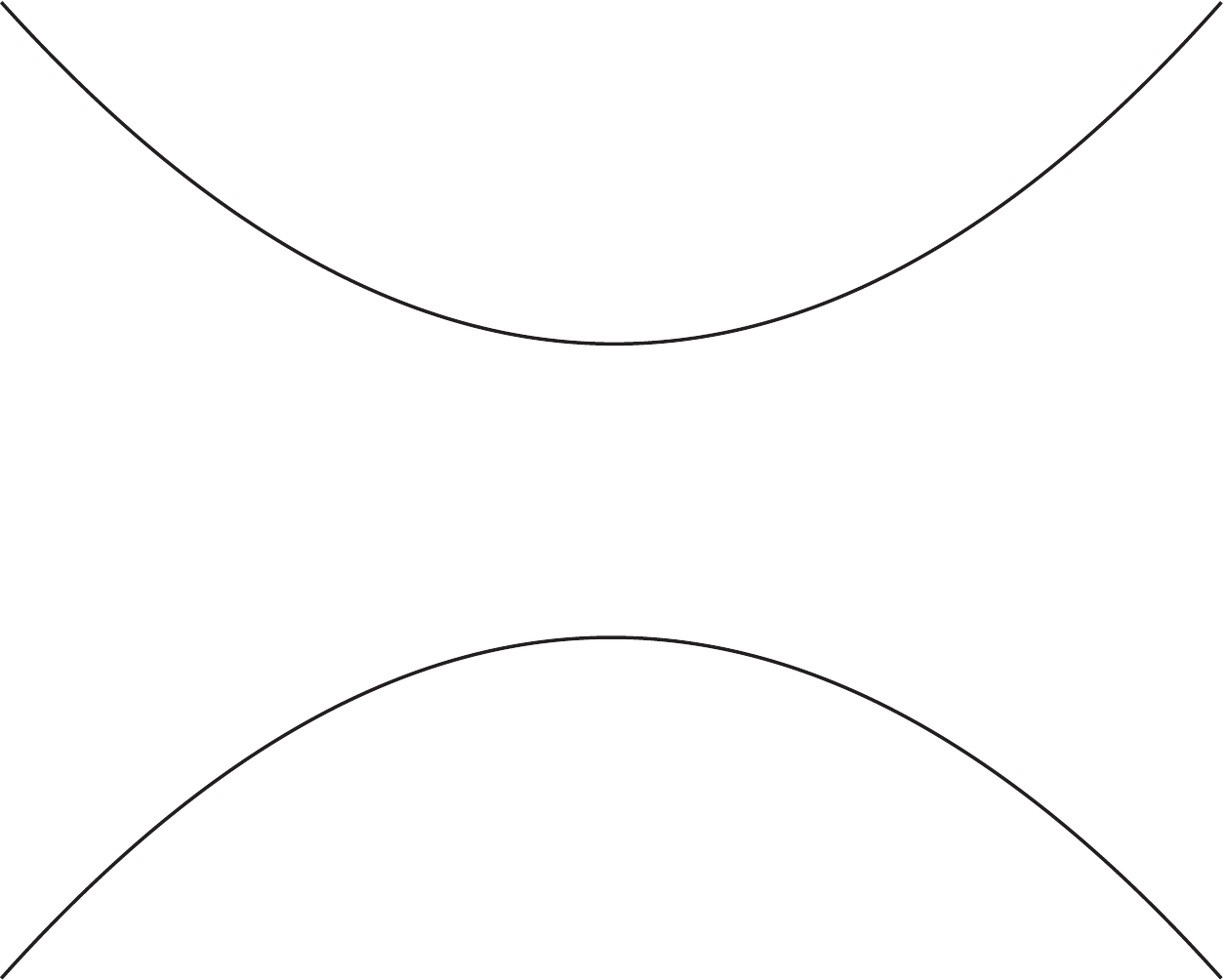}\end{minipage}\ +\ A\ 	 \begin{minipage}{0.2in}\includegraphics[width=0.2in]{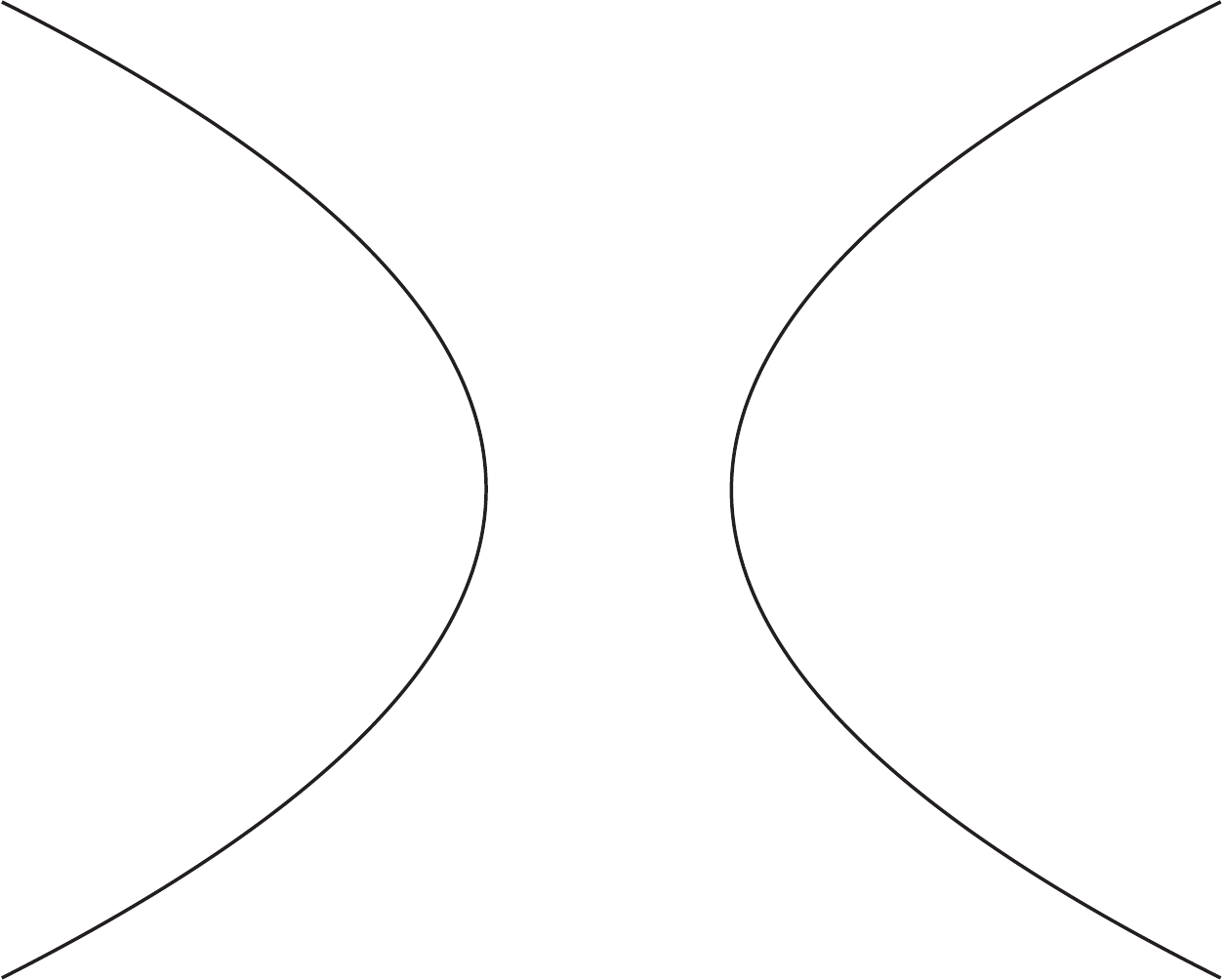}\end{minipage}$ .
	\end{enumerate}
\end{de}

Now, taking $F$ to be the 2-disk $D^2=I\times I$, we have:

\begin{de}
The $n^{th}$ Temperley-Lieb Algebra $TL_n$ is the linear skein
$\mathcal{S}(D^2,n)$, where $n$  means there are $n$ points specified in
$I\times\{0\}$ and $I\times\{1\}$ respectively.
\end{de}

It is well known that $TL_n$ has a basis, which consisting of non-crossing figures. We denote this basis by $\mathfrak{B}_n$.
Some special elements $\{1,e_1,...,e_{n-1}\}$ of the basis are shown in Figure \ref{f1}. 
As an algebra, $TL_n$ is generated by those special elements.

\begin{figure}[h]
\[ 1 =
            \begin{array}{c}
            \includegraphics[width=1in,height=1.2in]{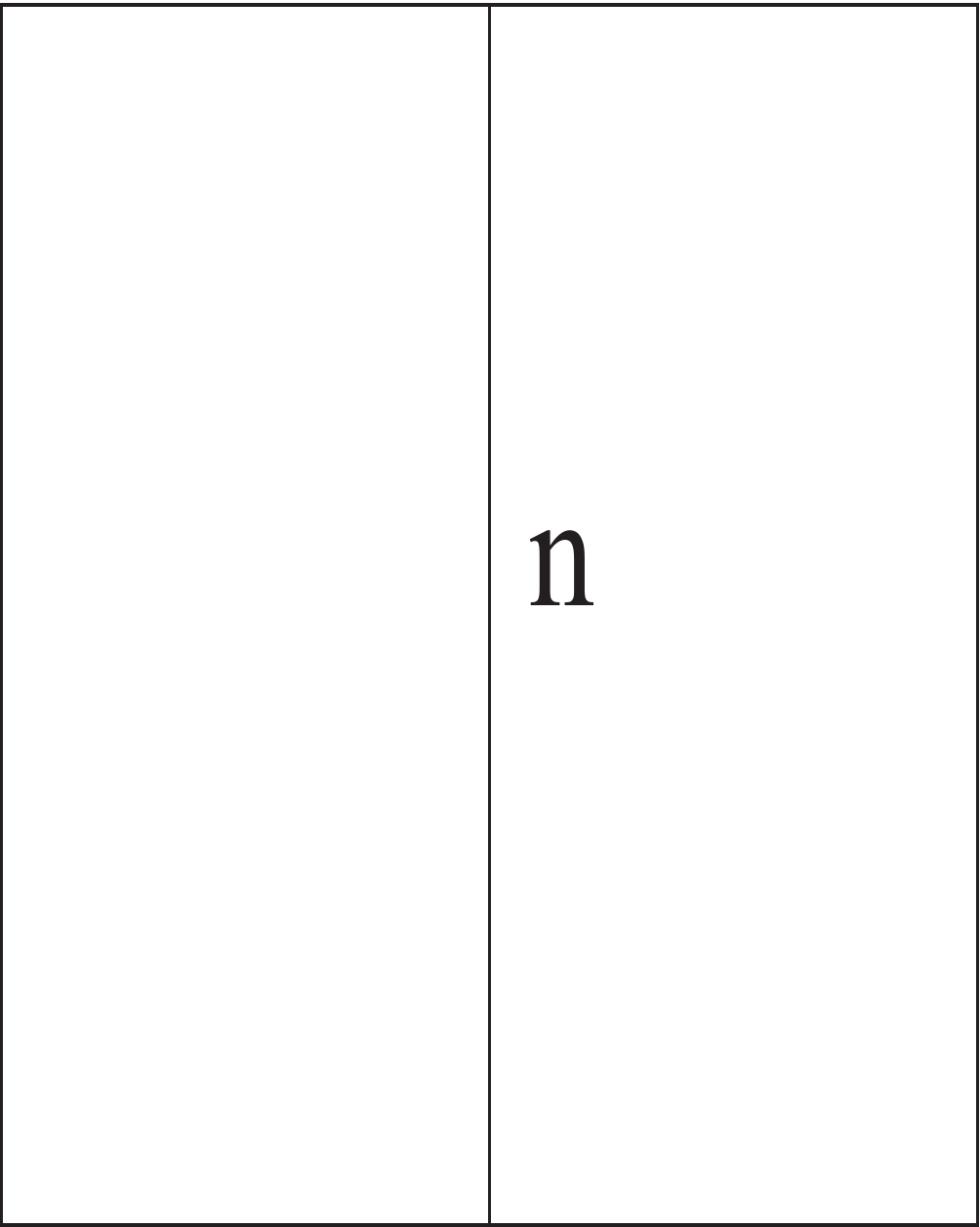}
            \end{array}
   e_i=
            \begin{array}{c}
            \includegraphics[width=1in,height=1.2in]{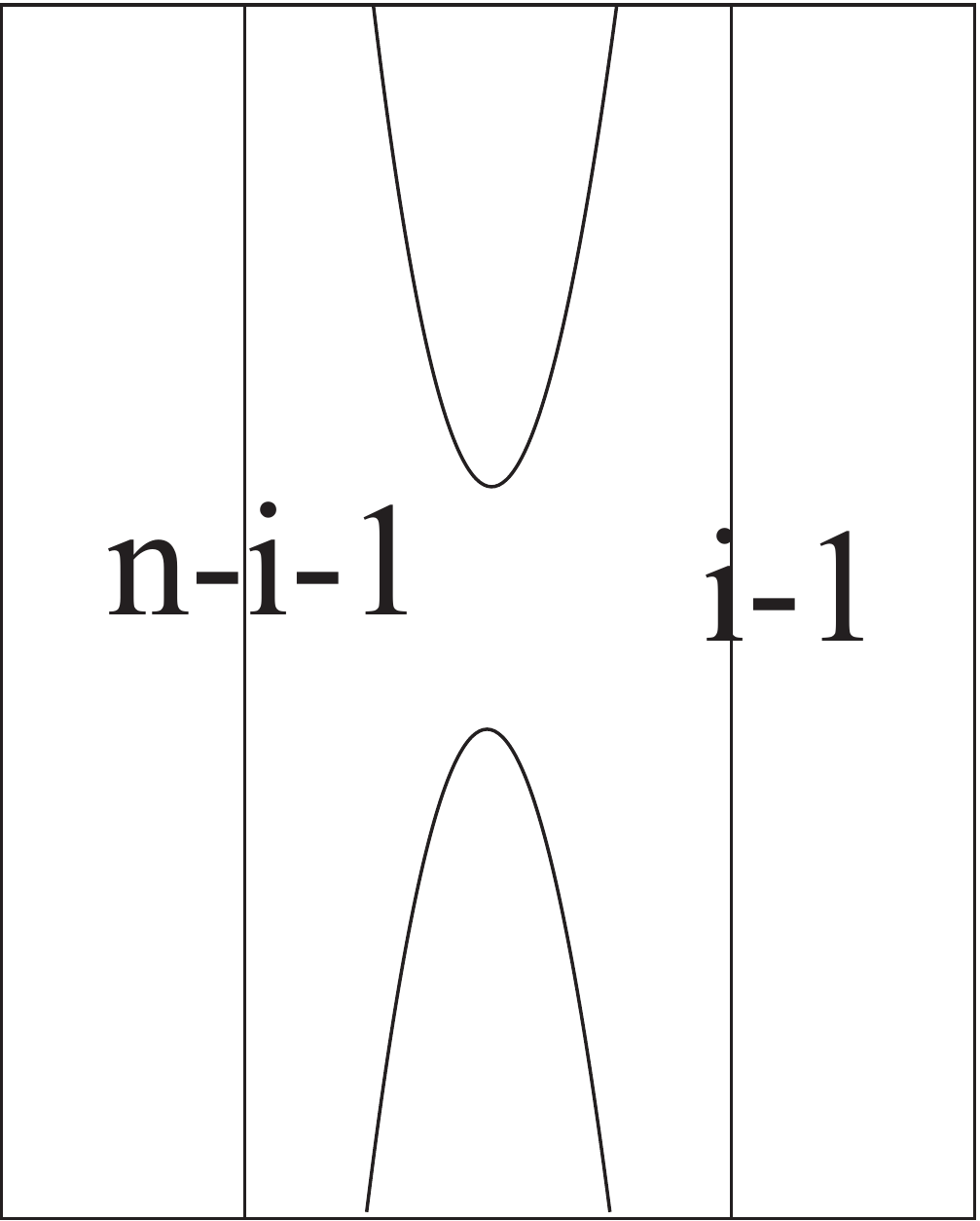}
            \end{array}
    \]
    \caption{The integer $i$ beside the arc means $i$ parallel copies of the arc.}
    \label{f1}
\end{figure}

A significant property of this algebra in quantum invariant theory is that there is a natural bilinear form on $TL_n$.
In \cite{L1}, Lickorish used this form to construct quantum invariants of 3-manifolds.
We construct this bilinear form with respect to the basis $\mathfrak{B}_n$ that we gave above:

\begin{de}
Define a map on $\mathfrak{B}_n\times\mathfrak{B}_n$ to $\Lambda$ as follows:
	\begin{center}	
	$G_n(s,r)=
	\begin{minipage}{0.2in}\includegraphics[width=0.2in]{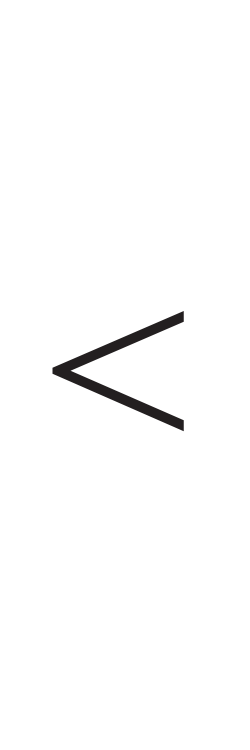}\end{minipage}
	\begin{minipage}{1in}\includegraphics[width=1in]{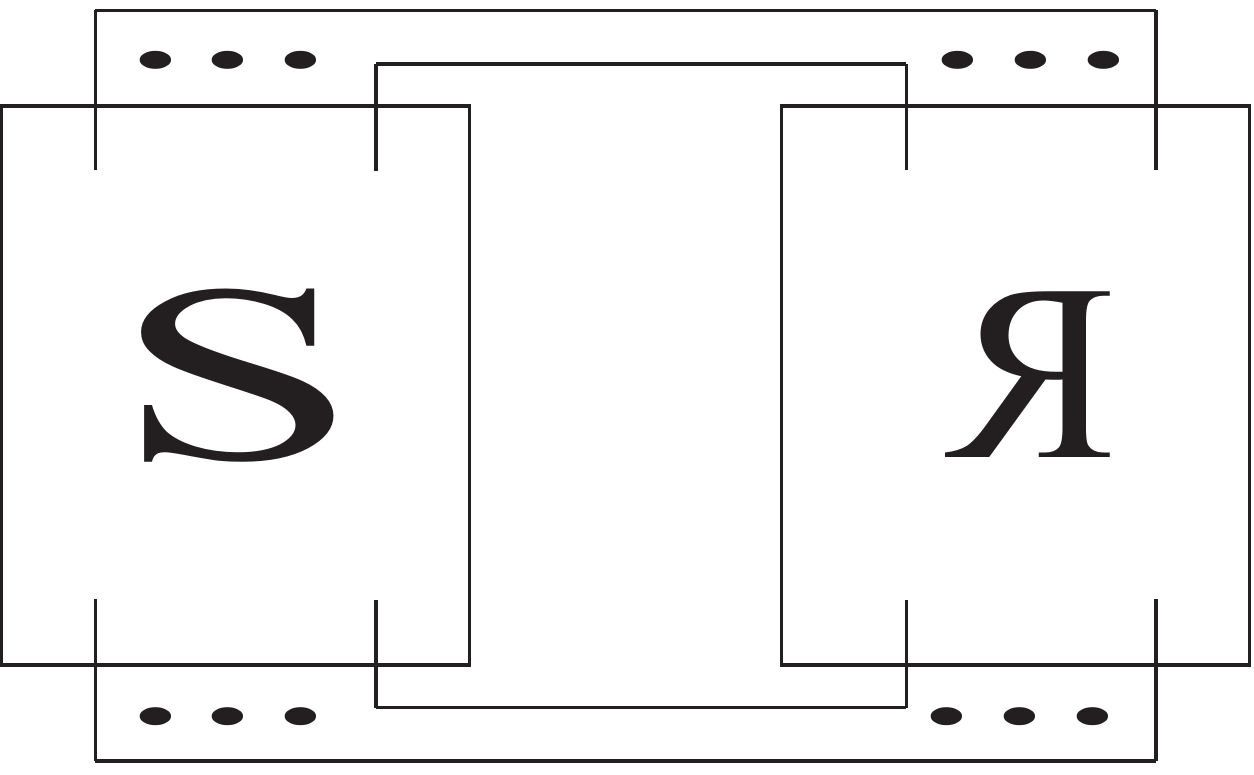}\end{minipage}
	\begin{minipage}{0.2in}\includegraphics[width=0.2in]{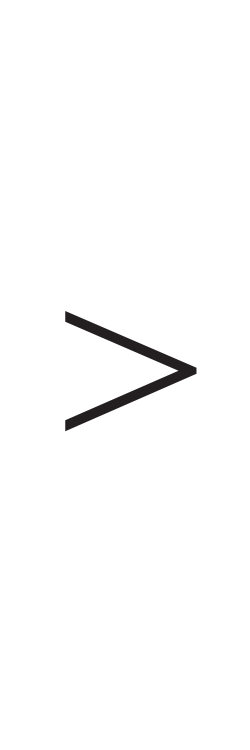}\end{minipage}$ ,
	\end{center}
where $s=\begin{minipage}{0.3in}\includegraphics[width=0.3in]{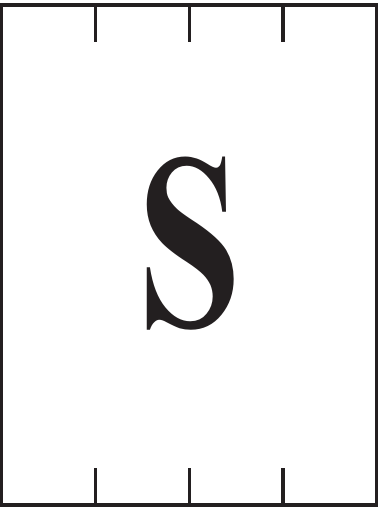}\end{minipage}$ and 
$r=\begin{minipage}{0.3in}\includegraphics[width=0.3in]{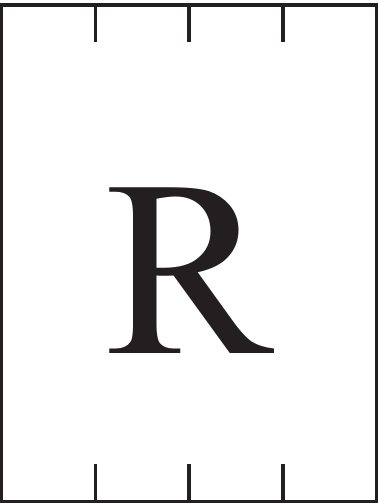}\end{minipage}$
are elements in $\mathfrak{B}_n$ and $<,>$ is the Kauffman bracket.
We extend this map to a bilinear form on $TL_n$, 
and still denote it by $G_n$.
We denote the determinant of $G_n$ with respect to $\mathfrak{B}_n$ by $\det(G_n)$.
\end{de}

In this paper, we give a simple proof of the determinant of this bilinear form with respect to the basis $\mathfrak{B}_n$,
which was also proved in \cite{FGG}.
The following is the main result.

\begin{thm}
\begin{equation}
\det(G_n)=\Delta_1^{c_n}\prod_{k=1}^n(\frac{\Delta_k}{\Delta_{k-1}})^{\alpha_k}
\notag
\end{equation}
where $\Delta_i=\frac{(-1)^i(A^{2(i+1)}-A^{-2(i+1)})}{A^2-A^{-2}}$, 
$c_n = \frac{1}{n+1}
   \left(
      \begin{array}{c}
       2n \\
       n
      \end{array}
   \right)
$, 
and
$\alpha_k =
        \left(
            \begin{array}{c}
            2n \\
            n-k
            \end{array}
        \right) -
        \left(
            \begin{array}{c}
            2n \\
            n-k-1
            \end{array}
        \right).
$
\end{thm}

\begin{rem}
From now on, we will use $\card$ to denote the cardinality of a set and $\det$ the determinant
of a matrix.
\end{rem}

\section{Properties of $TL_n$}
\label{prop}

In the 1990's, the properties of $TL_n$ were studied by Lickorish \cite{L2}, Masbaum-Vogel \cite{MV},
Kauffman-Lins \cite{KL} and some other people. Below we will summarize some results
on $TL_n$ that we will be using.

In this algebra, there is a sequence of idempotents, which are very
important in constructing 3-manifold invariants.  We will mainly use
these idempotents to construct a basis for $TL_n$. They are
defined as follows:

\begin{prop}\label{1}
There is a unique element $f_n\in TL_n$, called $n^{th}$ Jones-Wenzl
idempotent, such that
	\begin{enumerate}
		\item $f_ne_i=0=e_if_n$ for $1\leq i\leq n-1$;
		\item $(f_n-1)$ belongs to the subalgebra generated by $e_1,...,e_{n-1}$;
		\item $f_nf_n=f_n$.
	\end{enumerate}
\end{prop}

\begin{rem}
We can put a box on the segment to denote the idempotent.
But we will abbreviate the box from now on.
Hence, we put an $n$ beside the string to denote $n$ parallel strings with an idempotent inserted, if otherwise is not stated.
For example, we denote the figure on the left in Figure \ref{f2} by the figure on the right,
which will be used frequently in this paper.
\end{rem}

\begin{figure}[h]
  \centering
	\includegraphics[width=1in,height=1in]{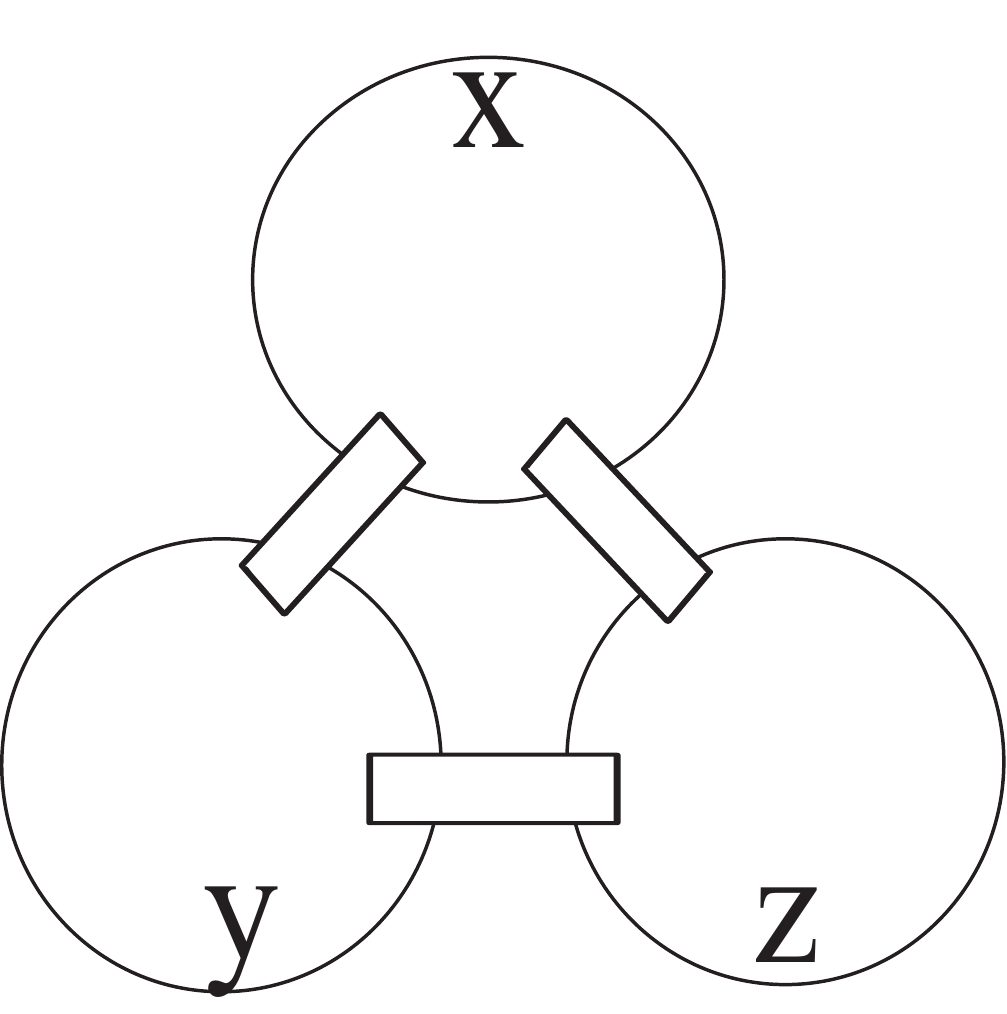}\hskip 0.5in
	\includegraphics[width=1in,height=1in]{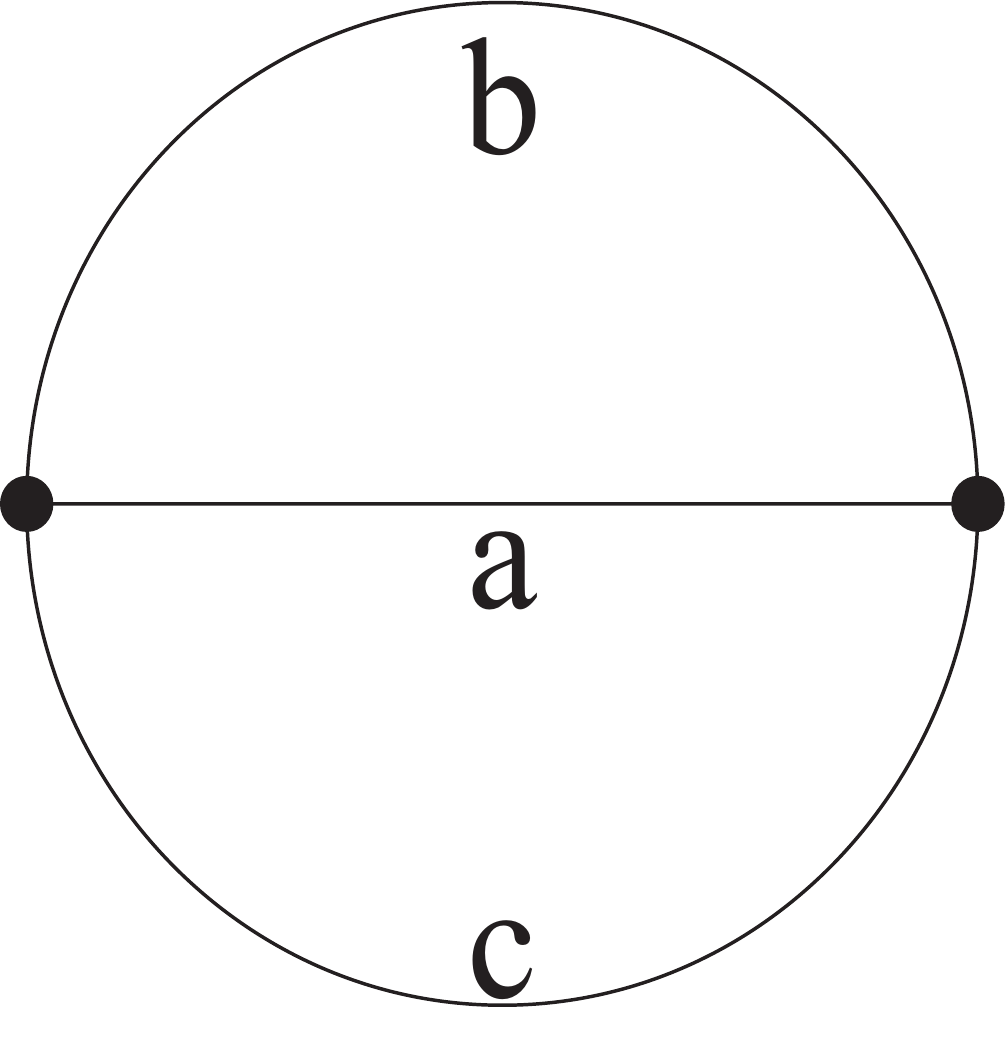}
	\caption{ The left figure lies in $S^2$. The right figure is an abbreviation of the left one, where $a=x+y,b=y+z,c=x+z$.
	We denote the value of this diagram in $\mathcal{S}(S^2)$by $\Theta(a,b,c)$.}

\label{f2}

\end{figure}

For the next property, we first set up some notation. Consider the
skein space of the disc $D$ with $a+b+c$ specified points on its
boundary. The points are partitioned into three sets of $a,b,c$
consecutive points. The effect of adding the idempotents
$f_a,f_b,f_c$ just outside every diagram in such a disc with
specified points is to map the skein space of the disc into a
subspace of itself. We denote this subspace by $T_{a,b,c}$.

\begin{de}

The triple $(a,b,c)$ of nonnegative integers will be called
admissible if $a+b+c$ is even, $a\leq b+c$ and $b\leq c+a$ and $c\leq a+b$.

\end{de}

\begin{prop}\label{3}
\[
dim(T_{a,b,c})=
\left\{
\begin{array}{l l}
  0 & \quad \text{if $a,b,c$ are not admissible};\\
  1 & \quad \text{if $a,b,c$ are admissible}.\\
\end{array} \right.
\]

\end{prop}

When $(a,b,c)$ is admissible, $T_{a,b,c}$ has a generator $g$ on the left in
Figure \ref{f5}. We usually denote it in a simple way by the diagram on the right
in the figure.

\begin{figure}[h]
	 \centering
   \includegraphics[width=1in,height=1in]{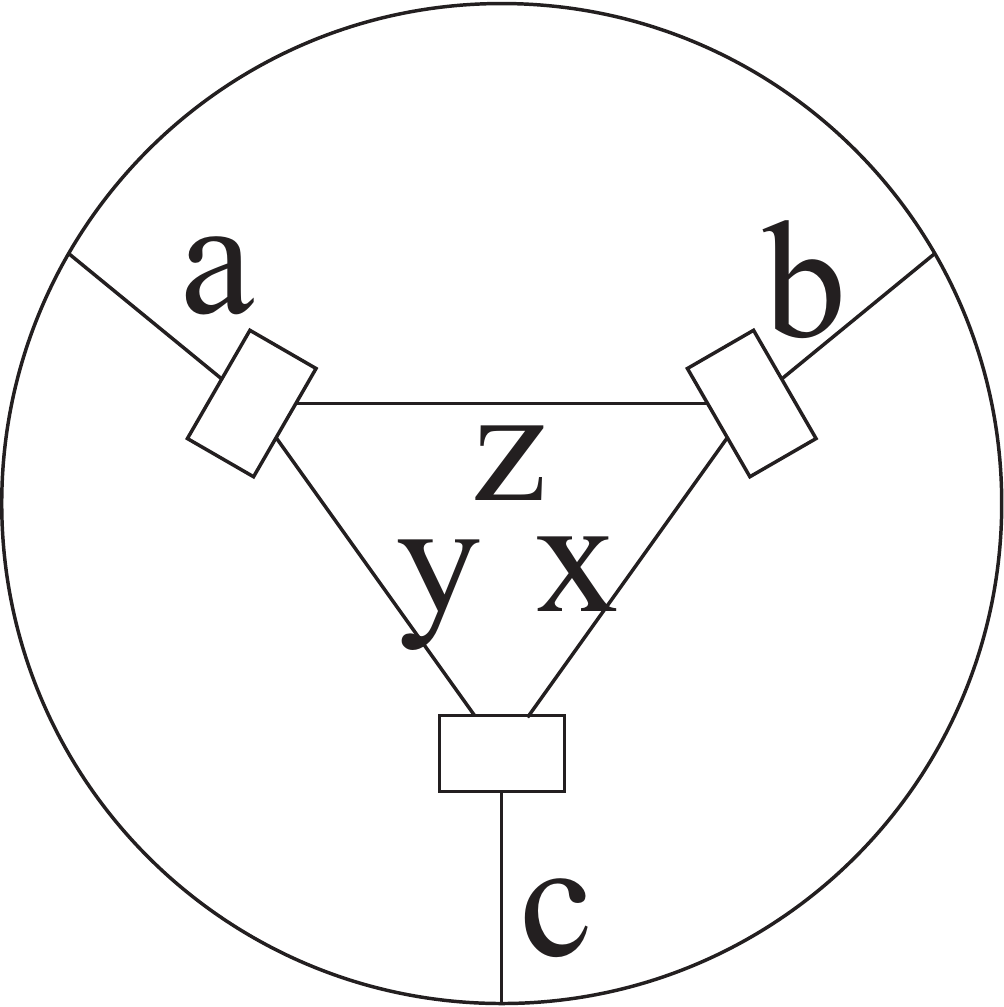}\hskip 0.5in
   \includegraphics[width=1in,height=1in]{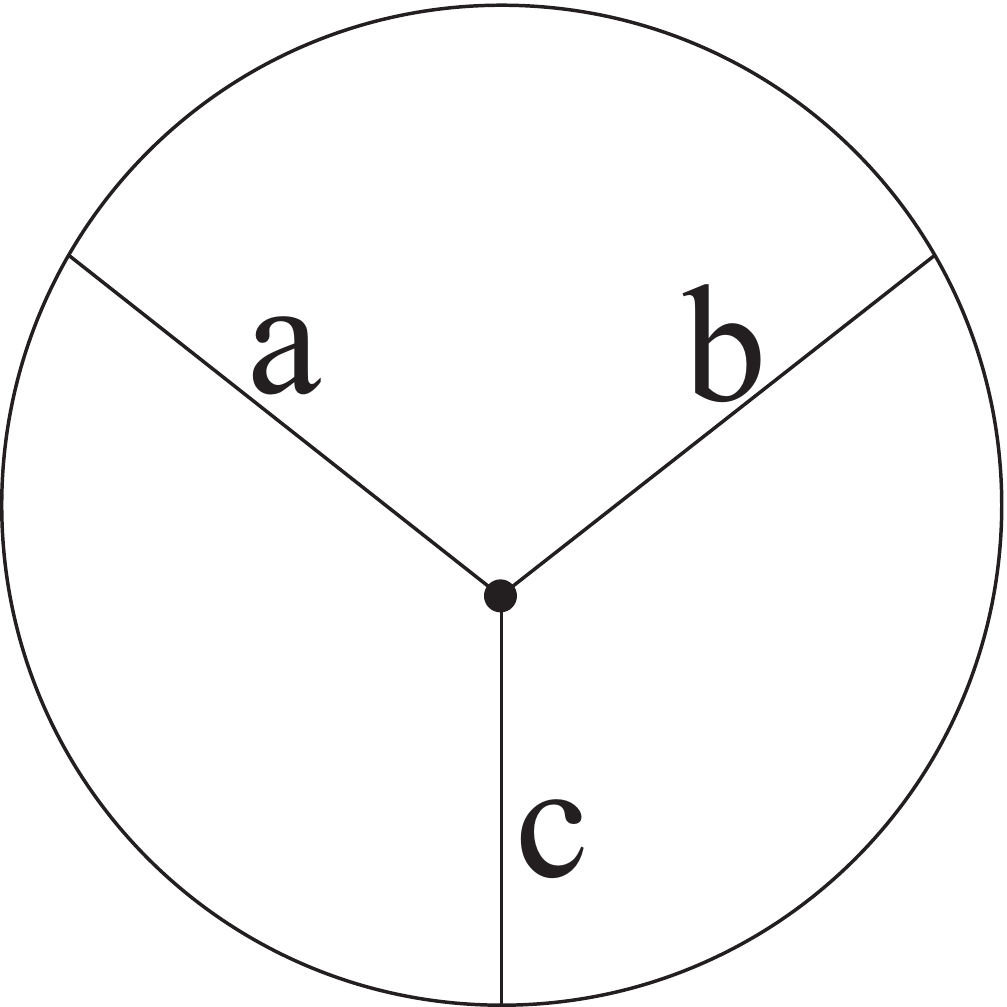}
   \caption{On the left is the generator of $T_{a,b,c}$. On the right is an abbreviation of the  generator.}
   \label{f5}
\end{figure}

\begin{prop}\label{4}
\begin{equation}
\begin{minipage}{1in}\includegraphics[width=1in]{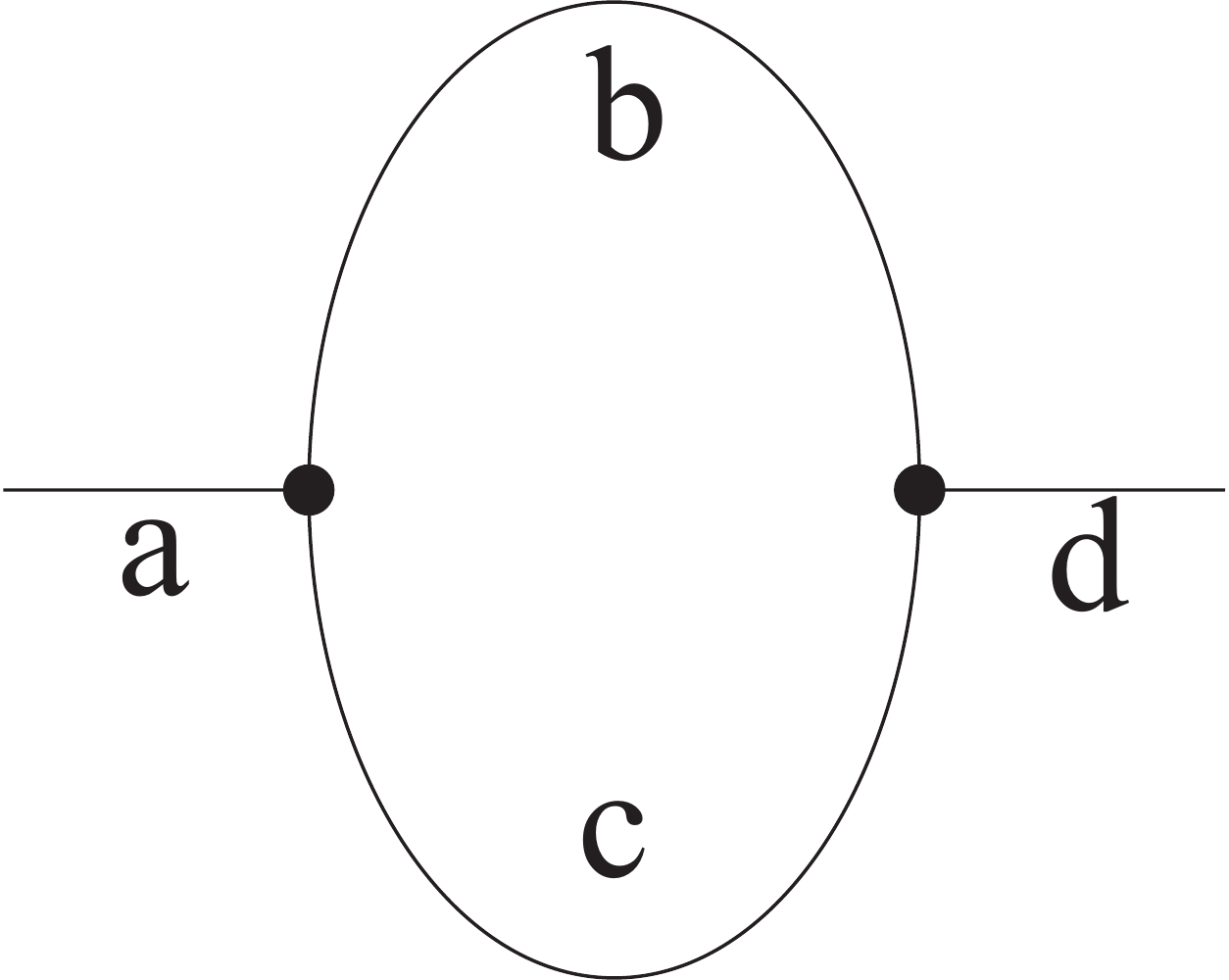}\end{minipage}
=\frac{\delta_{ad}\Theta(a,b,c)}{\Delta_a}\ 
\begin{minipage}{1in}\includegraphics[width=1in]{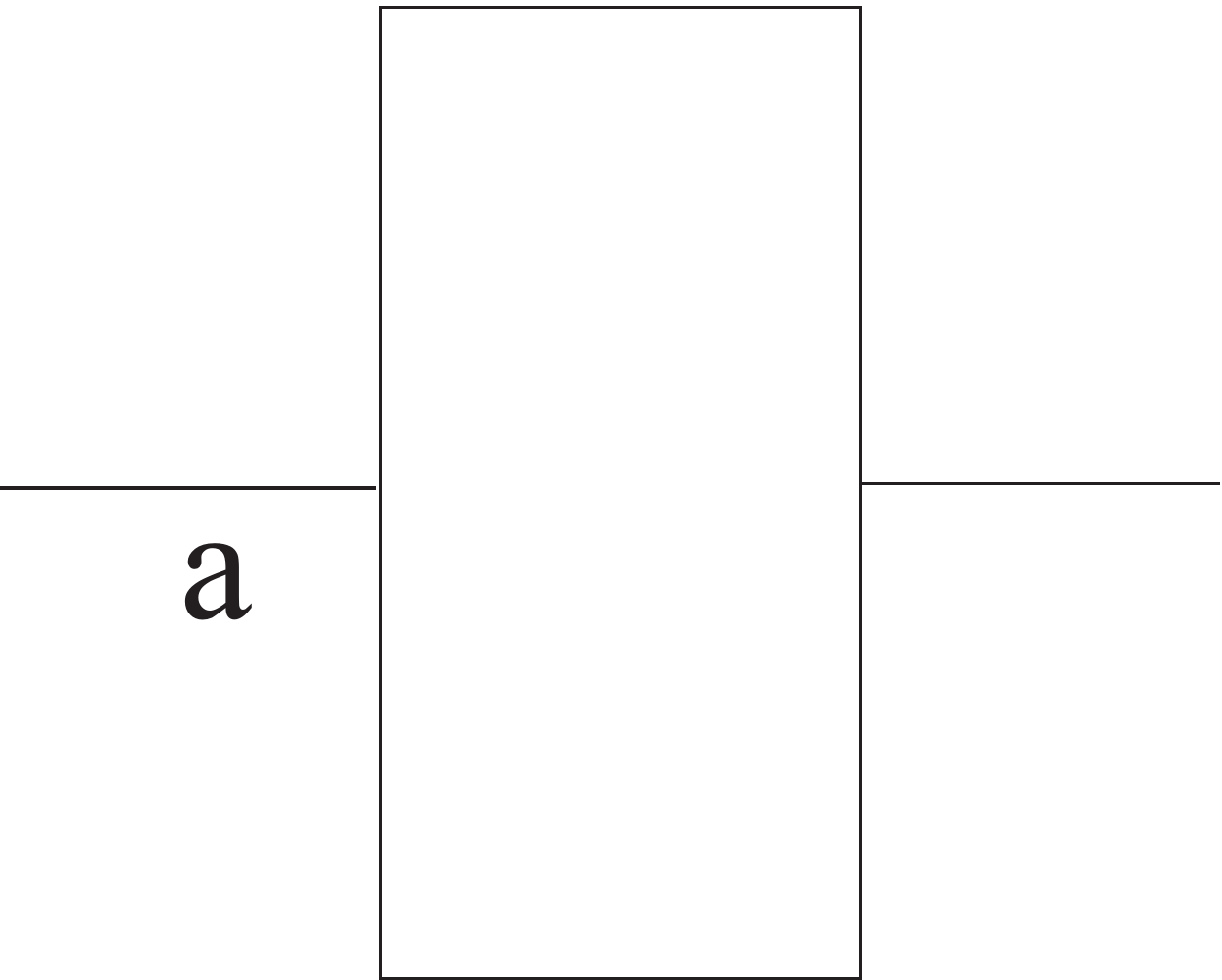}\end{minipage}\ ,
\notag
\end{equation}
where $\delta_{ad}$ is the Kronecker delta.
\end{prop}

Similarly, consider the skein space of the disc $D$ with $a+b+c+d$ specified points on its
boundary. The points are partitioned into four sets of $a,b,c,d$
consecutive points. The effect of adding the idempotents
$f_a,f_b,f_c,f_d$ just outside every diagram in such a disc with
specified points is to map the skein space of the disc into a
subspace of itself. We denote this subspace by $Q_{a,b,c,d}$.

\begin{prop}\label{5}
A base for $Q_{a,b,c,d}$ is the set of elements as in Figure \ref{f6}, where $j$ takes all values for which both $(a,b,j)$ and $(c,d,j)$ are admissible.
\end{prop}

\begin{figure}[h]
   \centering
   \includegraphics[width=1in,height=1in]{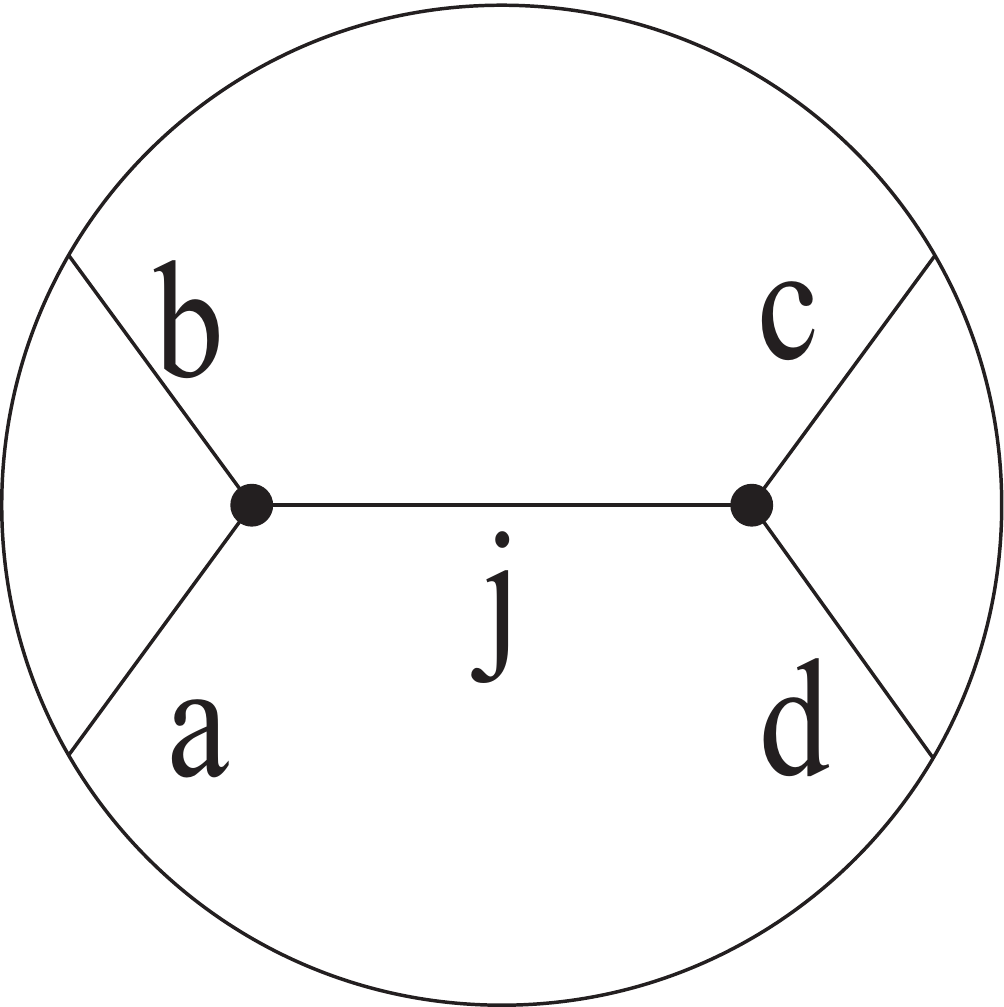}
   \caption{A basis element of $Q_{a,b,c,d}$}
   \label{f6}
\end{figure}

\begin{prop}\label{6}
\begin{equation}
\begin{minipage}{1in}\includegraphics[width=1in]{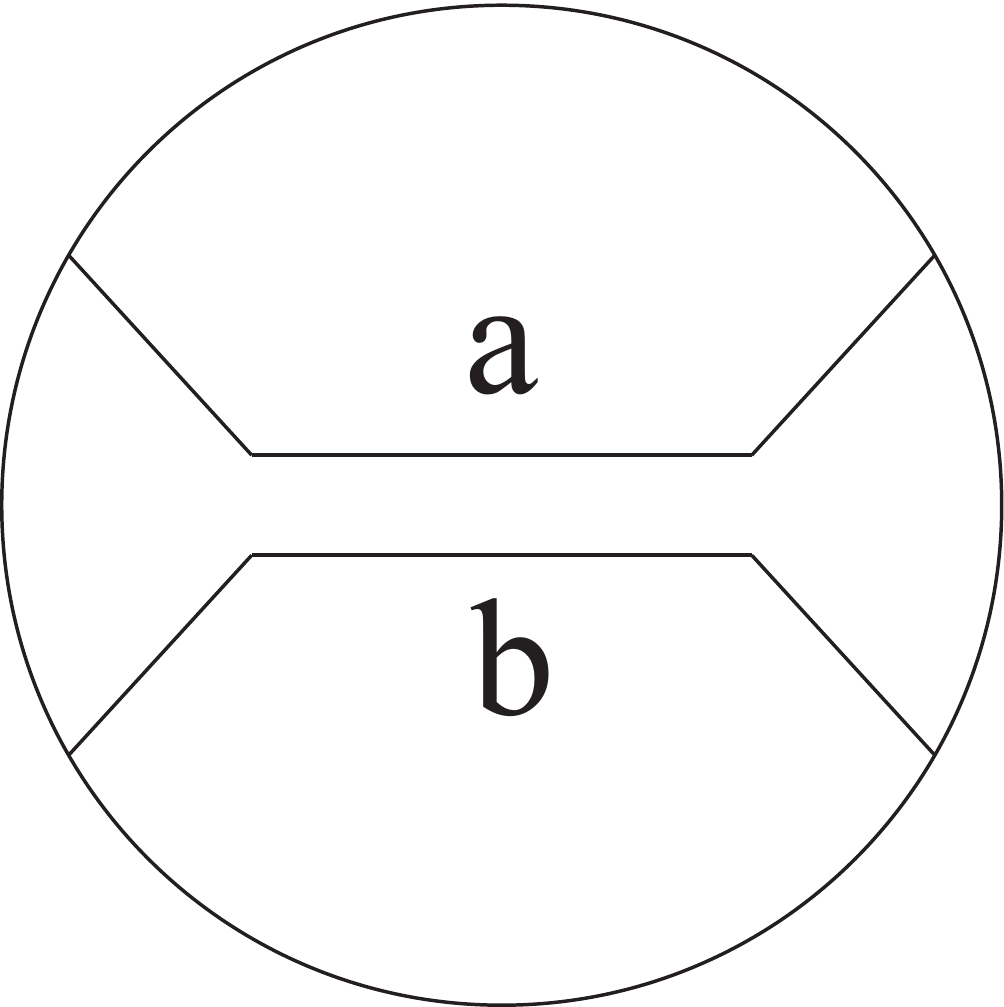}\end{minipage}
=\sum\frac{\Delta_j}{\Theta(a,b,j)}\ \begin{minipage}{1in}\includegraphics[width=1in]{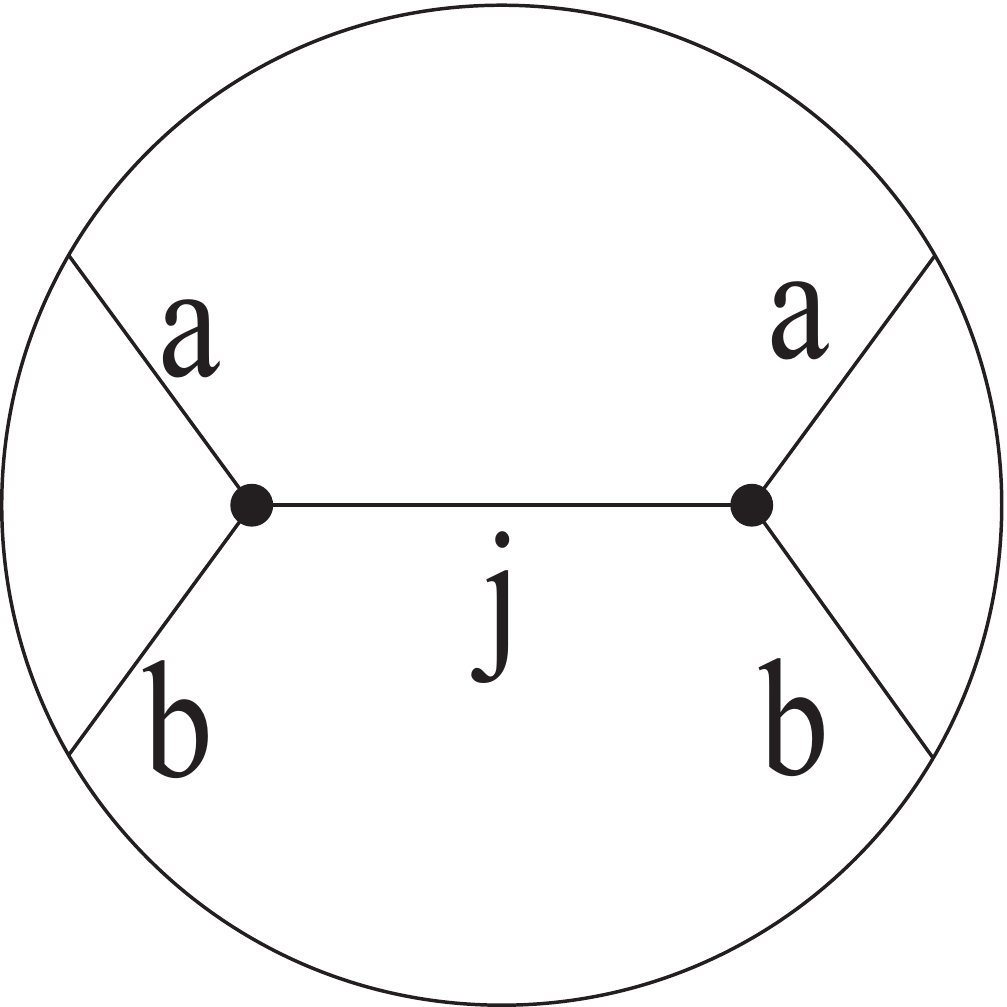}\end{minipage}
\notag
\end{equation}
where the summation runs over all $j$'s such that $(a,b,j)$ is admissible.
\end{prop}

\begin{rem}
We denote $\frac{\Theta(a,b,1)}{\Delta_a}$ by $\Gamma(b,a)$.
It is easy to see that $\Gamma(b,a)=0$ if $\|a-b\|>1$. 
Then Proposition \ref{4} becomes
\begin{equation}
\begin{minipage}{1in}\includegraphics[width=1in]{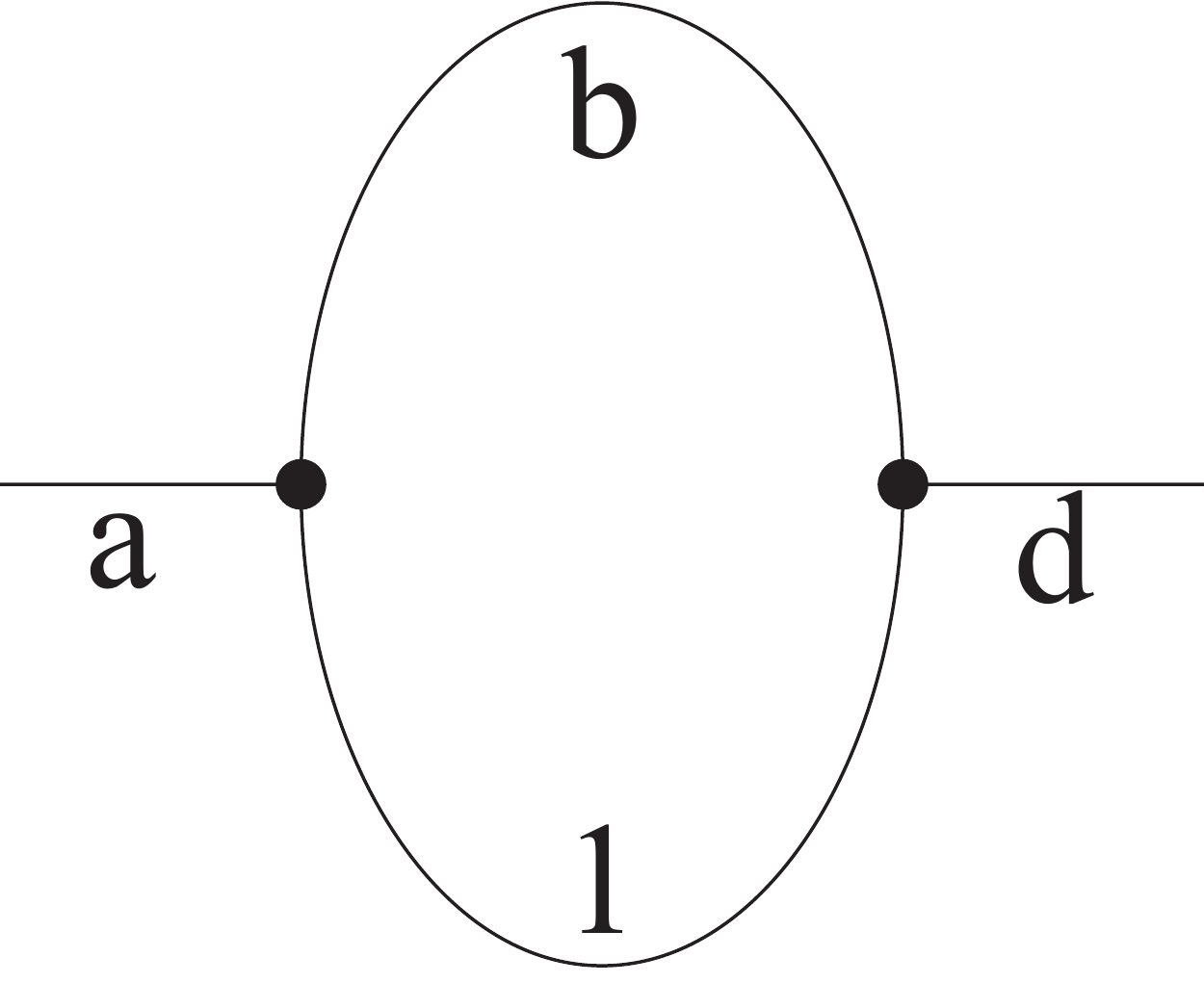}\end{minipage}
=\delta_{ad}\Gamma(b,a)\ 
\begin{minipage}{1in}\includegraphics[width=1in]{f_a.pdf}\end{minipage}\ .
\notag
\end{equation}
\end{rem}

\section{A Basis for the Temperley-Lieb Algebra from Jones-Wenzl Idempotents}
\label{JWidempotents}

Several bases of $TL_n$ have been given before by, for example, \cite{FGG} and \cite{GS}.
The idea of constructing the basis  in this paper is motivated by \cite{BHMV} and \cite{L3}. 
They constructed bases for modules associated to surfaces by a certain topological quantum field theory.
These bases were indexed by coloring of a trivalent graph in a handlebody.

\begin{de}\label{11}
Let $D_{a_1,...,a_{2n-1}}$ be the element of $TL_n$ in the Figure \ref{f3}
where $a_i$ satisfies:
	\begin{enumerate}
		\item $a_1=a_{2n-1}=1$;
		\item $a_i\in \mathbb{N}$ for all $i$;
		\item $\|a_i-a_{i-1}\|=1$ for all $i$.
	\end{enumerate}
Let $\mathfrak{A}_n$ be the collection of all n-tuples $(a_1,...,a_{2n-1})$ satisfying the above conditions, and let $\mathfrak{D}_n$ be the collection of all these $D$'s.
\end{de}

\begin{figure}[h]
\centering
\includegraphics[width=1in,height=1.2in]{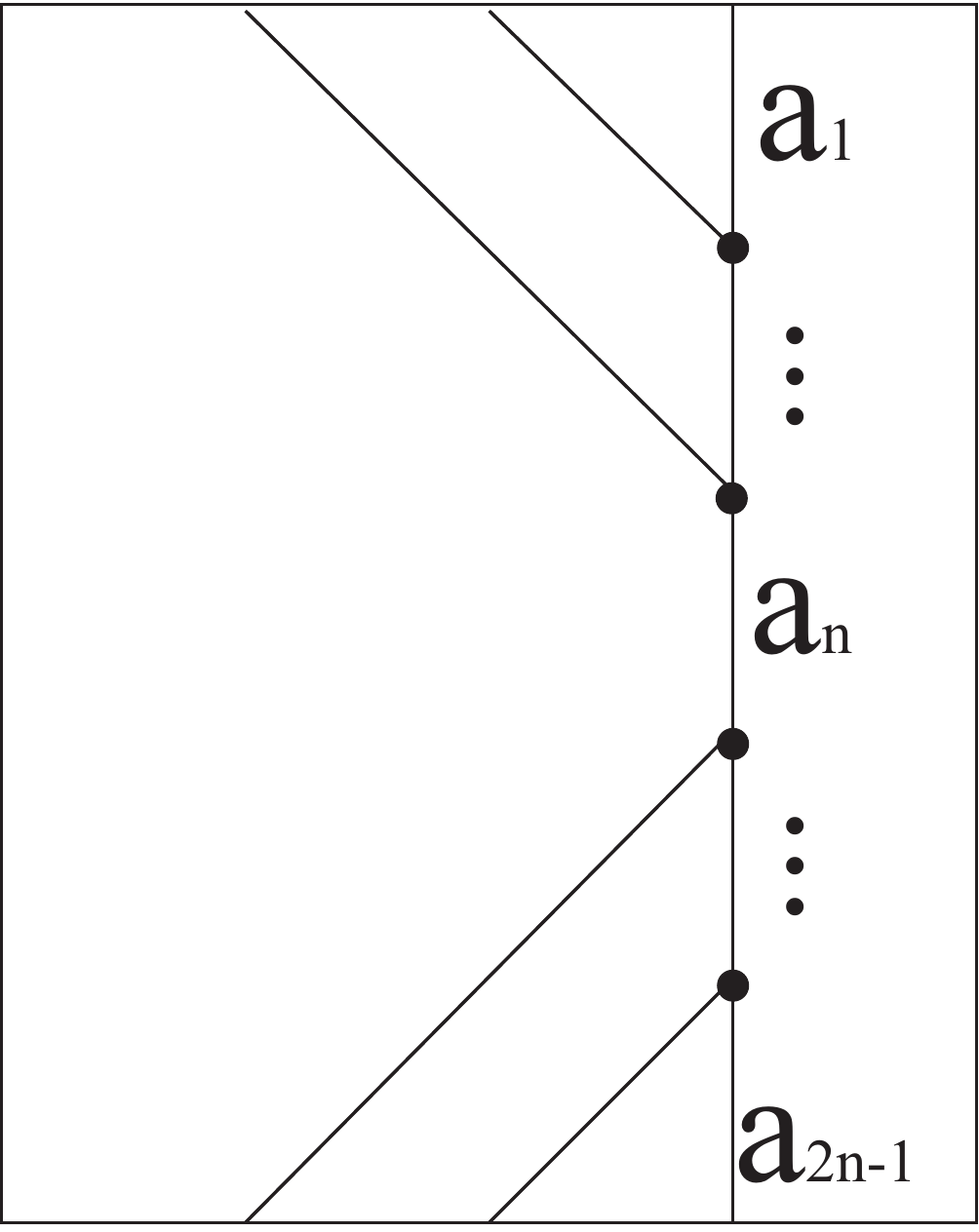}
\caption{Each triple point is admissible.}
\label{f3}
\end{figure}

\begin{lem}\label{7}
Suppose $(a_1,...,a_{2n-1})$ and $(b_1,...,b_{2n-1})$ satisfy all conditions above except $a_1=a_{2n-1}=1$. Then
\begin{eqnarray*}
&&\langle D_{a_1,...,a_{2n-1}},D_{b_1,...,b_{2n-1}}\rangle \nonumber \\
&=&\delta_{a_1b_1}...\delta_{a_{2n-1}b_{2n-1}}\Gamma(a_1,a_2)\Gamma(a_2,a_3)\dots\Gamma(a_{2n-2},a_{2n-1})\Delta_{a_{2n-1}}\ .
\end{eqnarray*}
\end{lem}

\begin{proof}
We prove the formula by induction.\\
When $n=2$, by direct computation,
\begin{eqnarray*}
&&\langle D_{a_1,a_2,a_3},D_{b_1,b_2,b_3}\rangle \nonumber \\
&=&\delta_{a_1b_1}\Gamma(a_1,a_2)\delta_{a_2b_2}\Gamma(a_2,a_3)\delta_{a_3b_3}\Delta_{a_3}.
\end{eqnarray*}
Thus the formula is true for $n=2$.
Now suppose the formula is true for $n=k-1$ and let $n=k$.
\begin{eqnarray*}
&&\langle D_{a_1,...,a_{2n-1}},D_{b_1,...,b_{2n-1}}\rangle \nonumber \\
&=&\delta_{a_1b_1}\Gamma(a_1,a_2)\delta_{a_{2n-1}b_{2n-1}}\Gamma(a_{2n-2},a_{2n-1})\langle D_{a_2,...,a_{2n-2}},D_{b_2,...,b_{2n-2}}\rangle \nonumber \\
&=&\delta_{a_1b_1}\Gamma(a_1,a_2)\delta_{a_{2n-1}b_{2n-1}}\Gamma(a_{2n-2},a_{2n-1})\times \nonumber \\ &&\delta_{a_2b_2}...\delta_{a_{2n-2}b_{2n-2}}\Gamma(a_2,a_3)\Gamma(a_3,a_4)\Gamma(a_{2n-3},a_{2n-2})\Delta_{a_{2n-2}} \nonumber \\
&=&\delta_{a_1b_1}...\delta_{a_{2n-1}b_{2n-1}}\Gamma(a_1,a_2)\Gamma(a_2,a_3)\dots\Gamma(a_{2n-2},a_{2n-1})\Delta_{a_{2n-1}}.
\end{eqnarray*}
Thus the formula holds for $n=k$.
Hence, by induction, the formula holds.
\end{proof}

\begin{lem}\label{8}
The elements of $\{D_{a_1,...,a_{2n-1}}\}$ are orthogonal in $TL_n$, and so are linearly independent.
\end{lem}
\begin{proof}
This follows from Lemma \ref{7}.
\end{proof}

Now, we are going to prove that the elements of $\mathfrak{D}_n$ generate $TL_n$. 
This can be proved using induction and Proposition \ref{1} and \ref{3}. 
For variety, we give an alternative proof.

\begin{lem}\label{9}
Each element of $\{1,e_1,...,e_{n-1}\}$ can be expressed as a linear combination of $D$'s in $\mathfrak{D}_n$.
\end{lem}
\begin{proof}
We prove the lemma by induction and Proposition \ref{6}.\\
It is easy to see that the lemma is true for $\mathfrak{D}_1,\mathfrak{D}_2$.\\
Suppose the lemma is true for $\mathfrak{D}_{n-1}$, 
we need show that it is true for $\mathfrak{D}_n$.\\
For $x\in\{1,e_2,e_3,...,e_{n-1}\}$,
we can obtain the result as in Figure \ref{fusion1}.
The proof for $e_1$ is similar except at the second equality, 
we use Proposition \ref{6} for each turn-back, see Figure \ref{fusion2}.
\end{proof}
\begin{figure}[h]
\includegraphics[width=1in,height=1.2in]{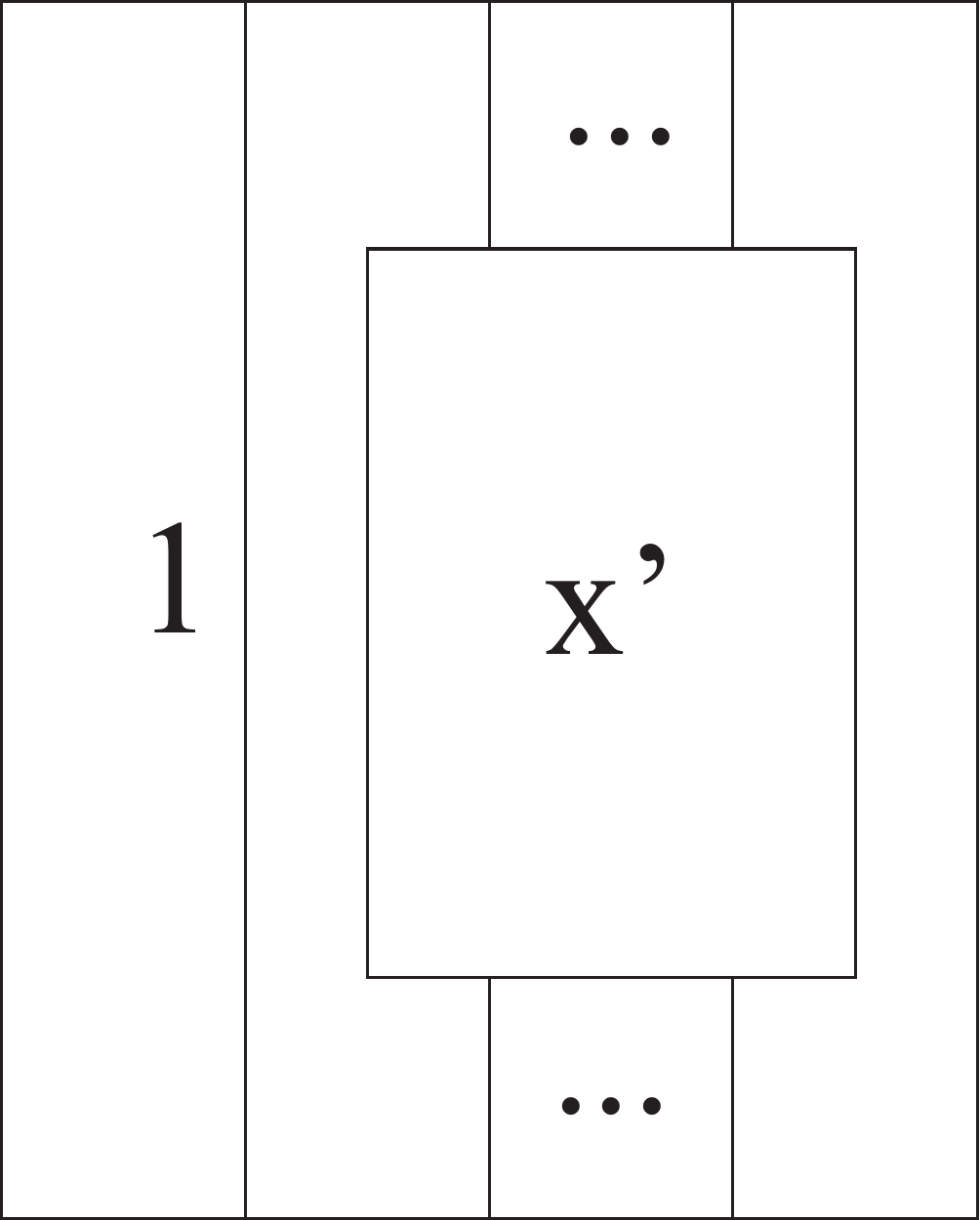}
\includegraphics[width=0.15in,height=1.2in]{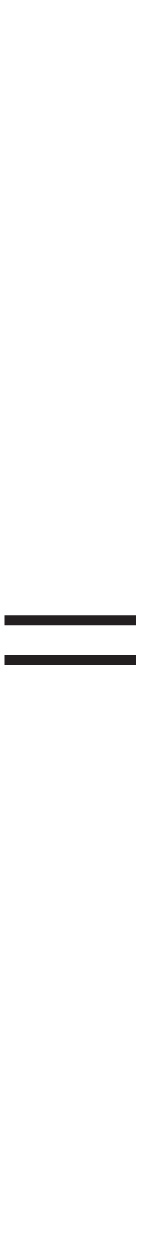}
\includegraphics[width=0.15in,height=1.2in]{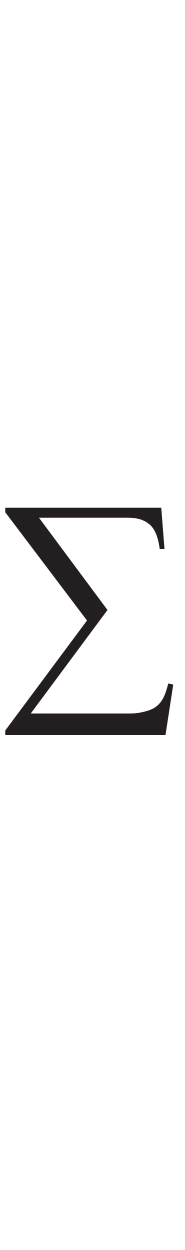}
\includegraphics[width=0.5in,height=1.2in]{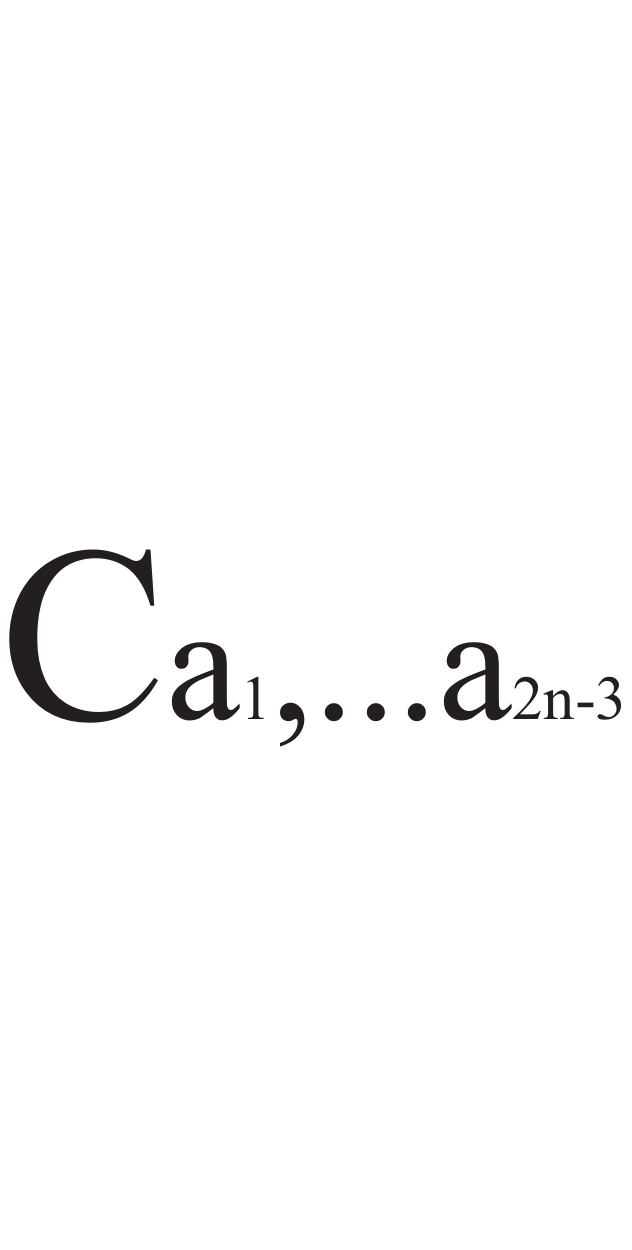}
\includegraphics[width=1in,height=1.2in]{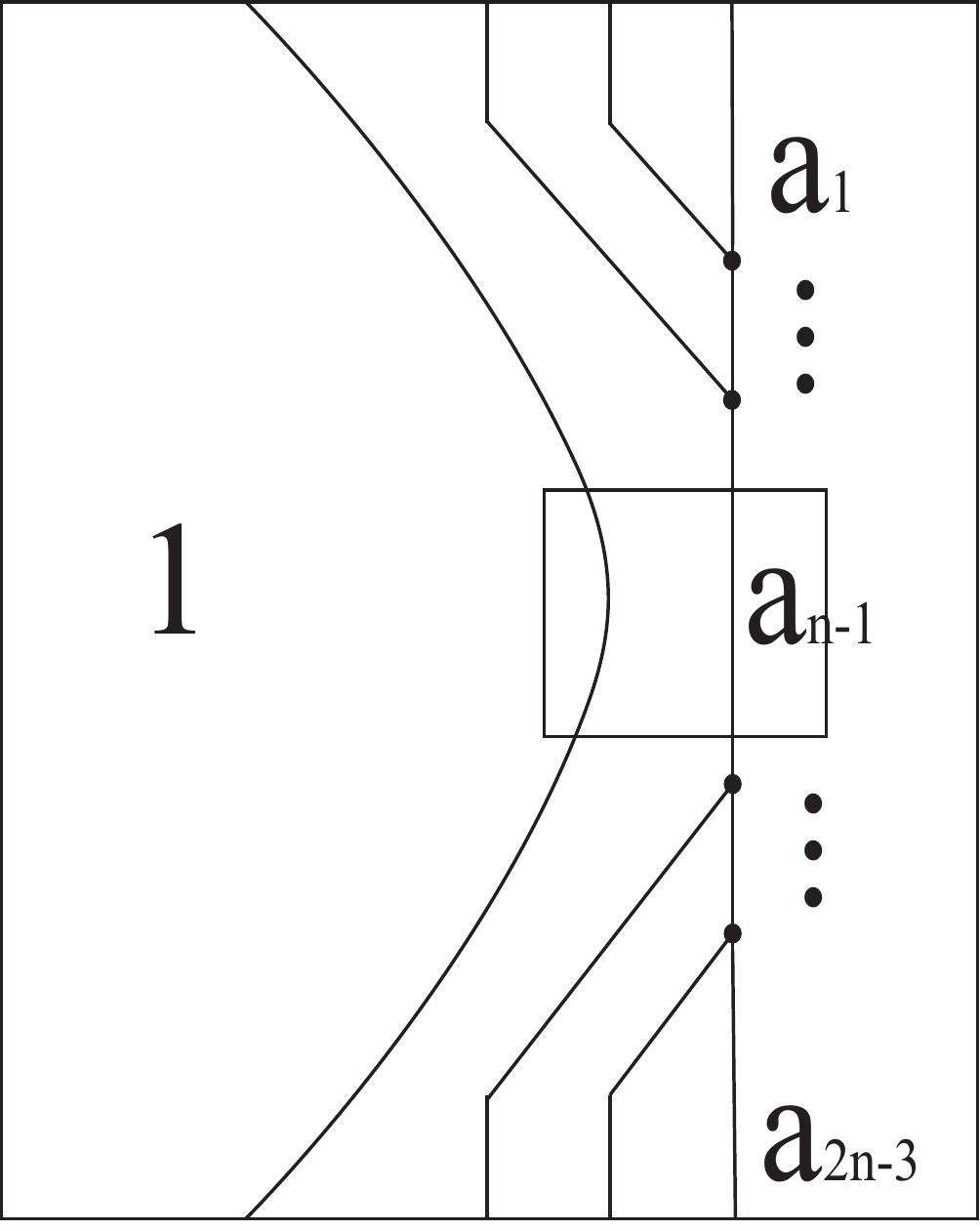}
\includegraphics[width=0.15in,height=1.2in]{equal.pdf}
\includegraphics[width=0.15in,height=1.2in]{sum.pdf}
\includegraphics[width=0.5in,height=1.2in]{coeff1.pdf}
\includegraphics[width=0.15in,height=1.2in]{sum.pdf}
\includegraphics[width=0.15in,height=1.2in]{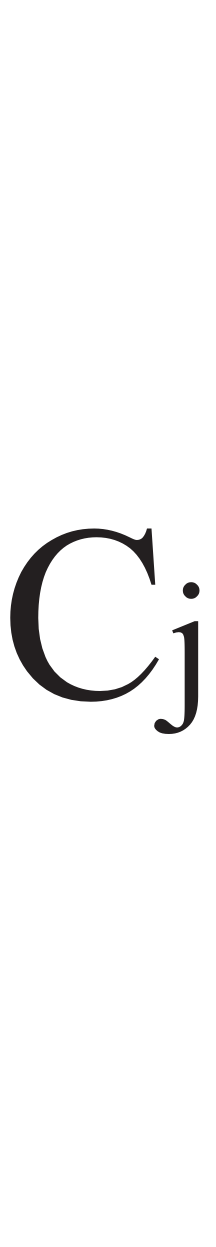}
\includegraphics[width=1in,height=1.2in]{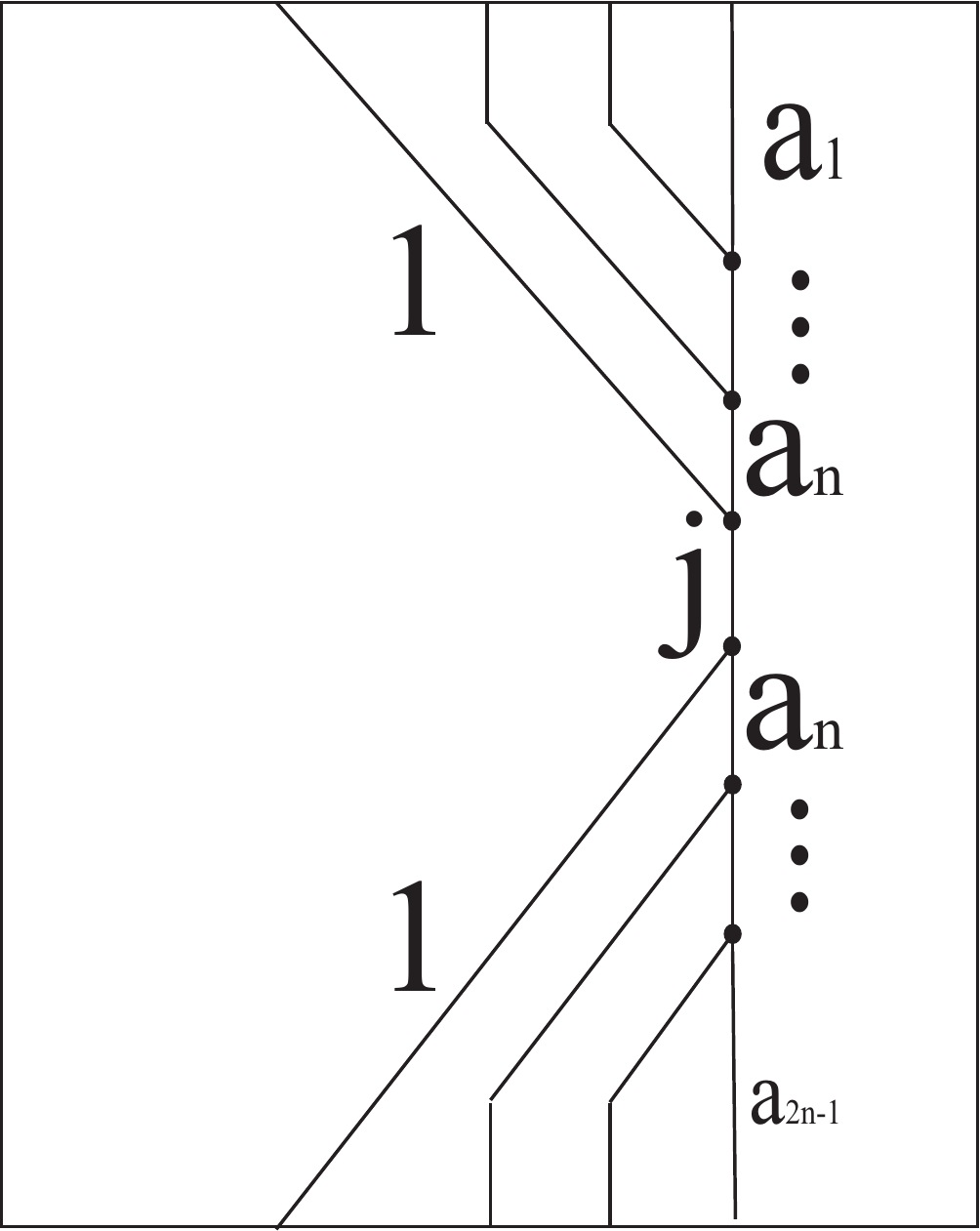}
\caption{$x's$ is a generator for $TL_{n-1}$ by deleting the first arc in $x$. 
The first equality is from induction step. 
The second equality is from Proposition \ref{6}.}
\label{fusion1}
\end{figure}

\begin{figure}[h]
\includegraphics[width=1in,height=1.2in]{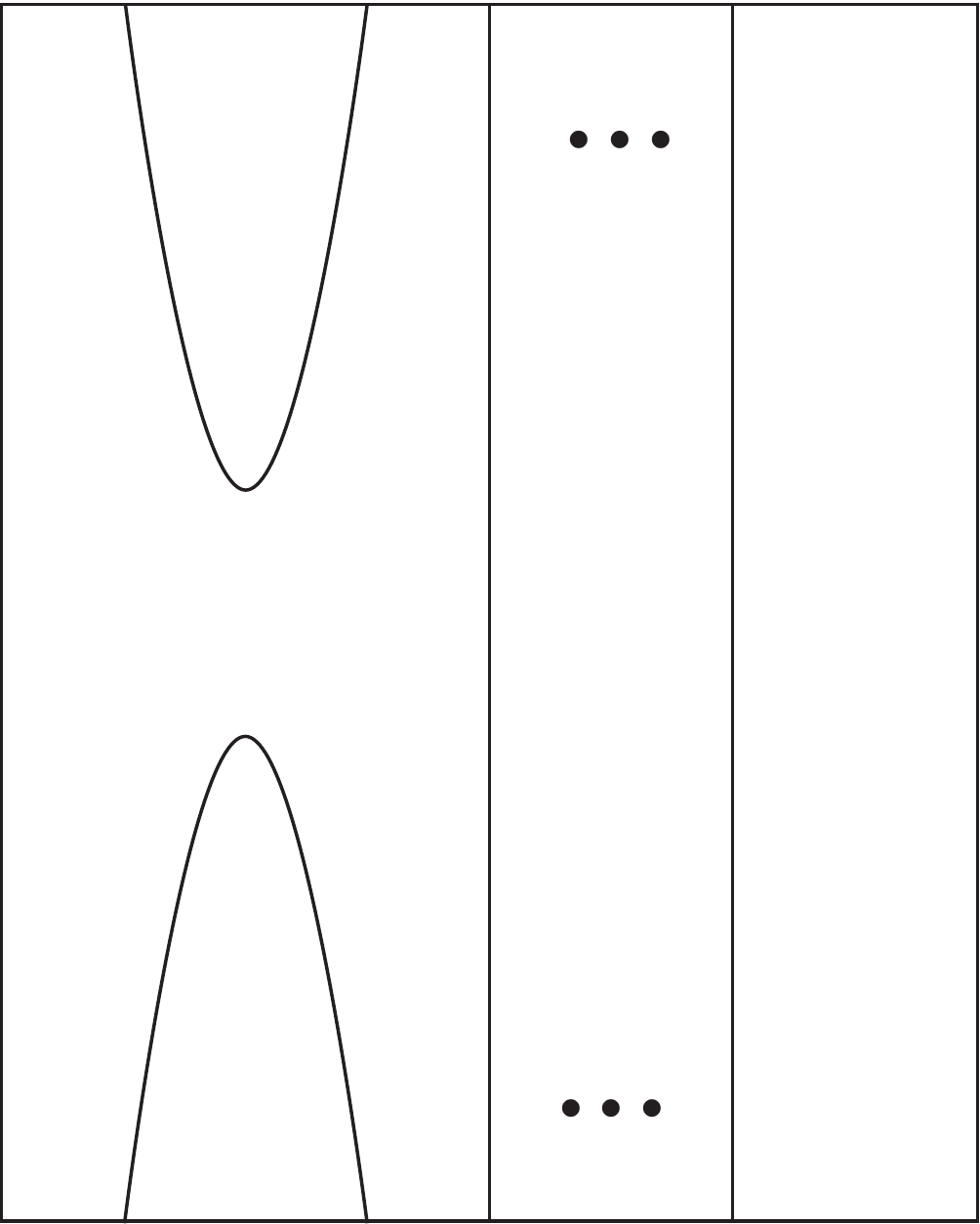}
\includegraphics[width=0.15in,height=1.2in]{equal.pdf}
\includegraphics[width=0.15in,height=1.2in]{sum.pdf}
\includegraphics[width=0.5in,height=1.2in]{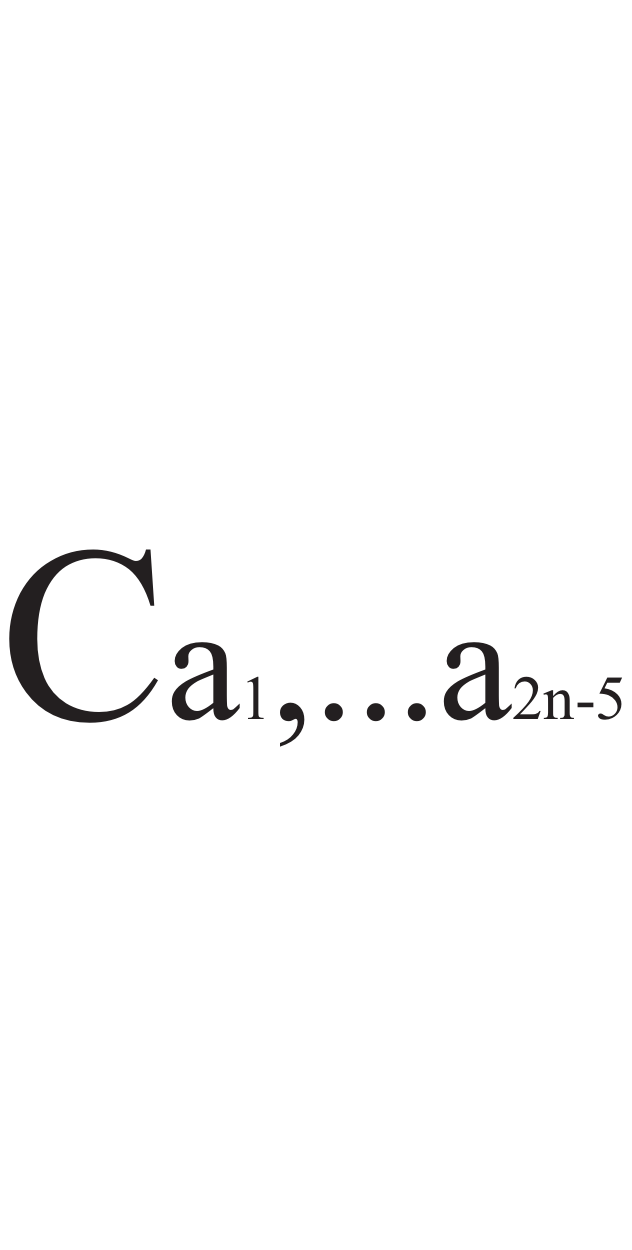}
\includegraphics[width=1in,height=1.2in]{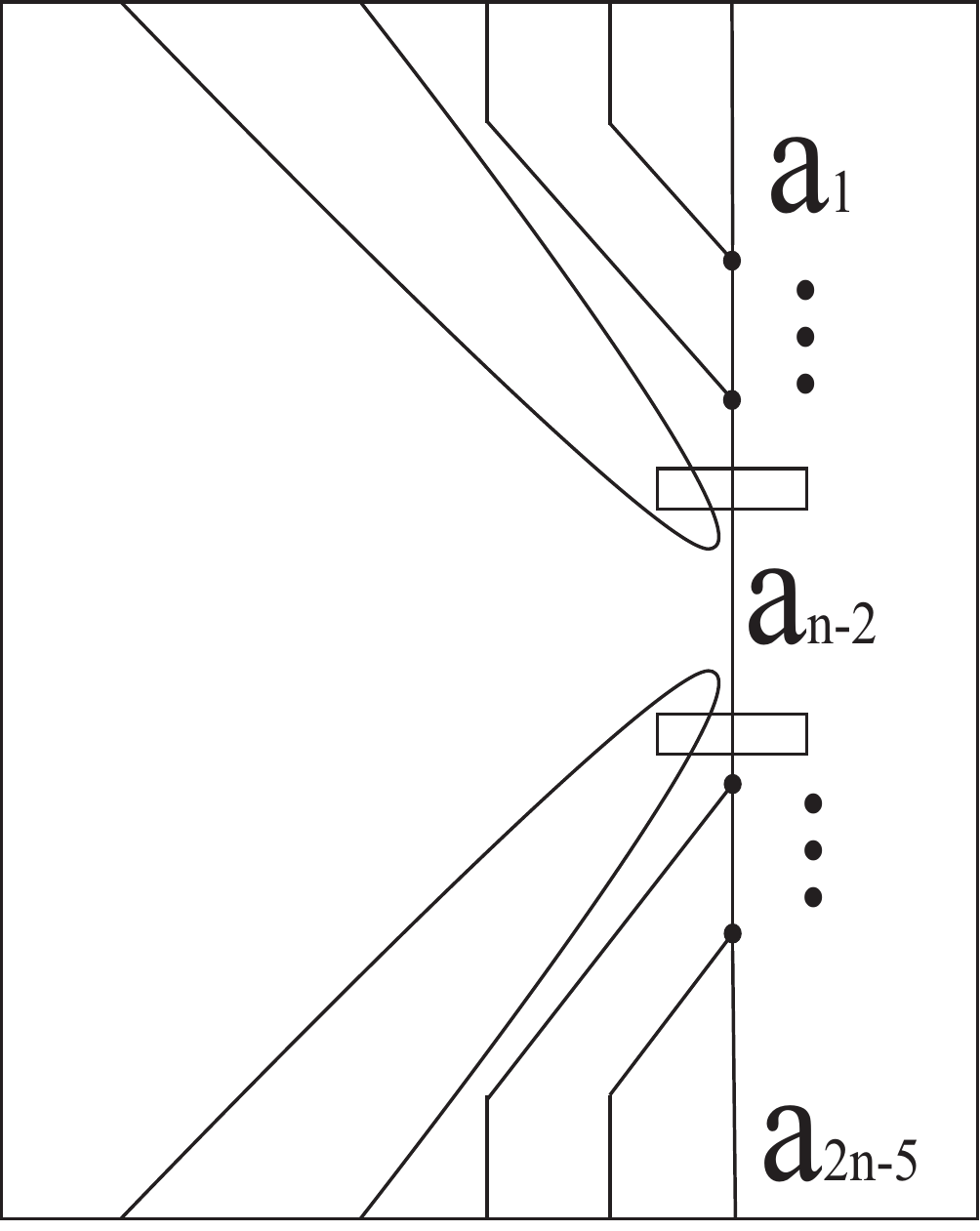}
\includegraphics[width=0.15in,height=1.2in]{equal.pdf}
\includegraphics[width=0.15in,height=1.2in]{sum.pdf}
\includegraphics[width=0.5in,height=1.2in]{coeff2.pdf}
\includegraphics[width=0.15in,height=1.2in]{sum.pdf}
\includegraphics[width=0.2in,height=1.2in]{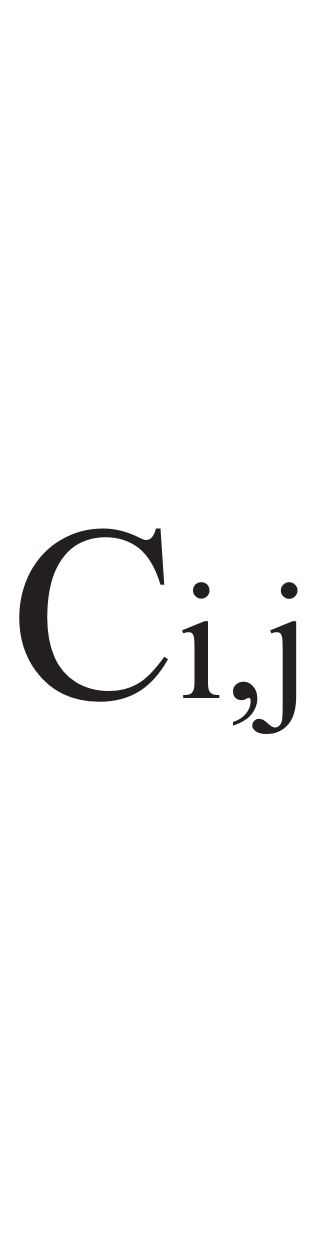}
\includegraphics[width=1in,height=1.2in]{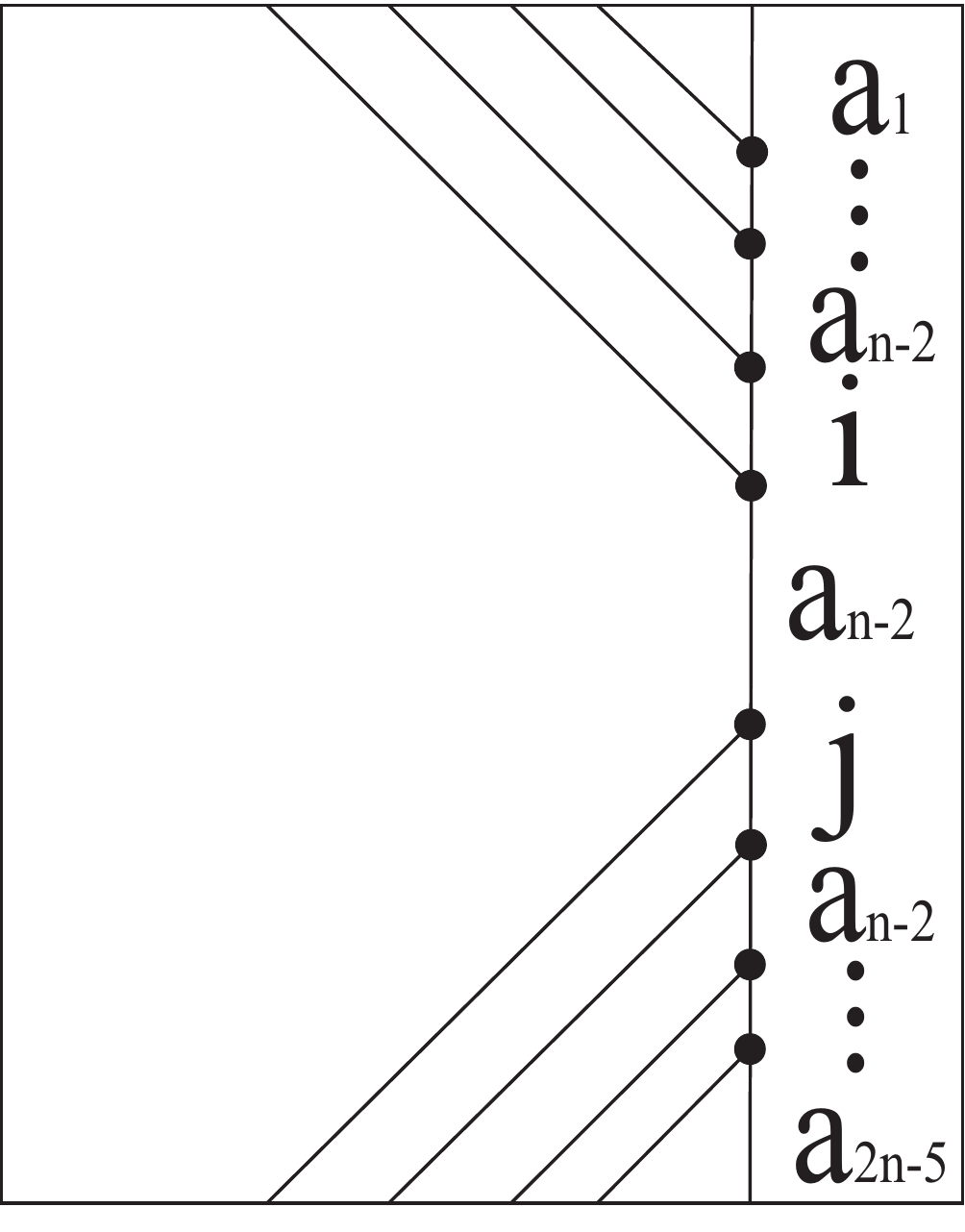}
\caption{The proof is the same as in Figure \ref{fusion1} except we need use Proposition \ref{6} twice.}
\label{fusion2}
\end{figure}

\begin{lem}
$\mathfrak{D}_n$ is a basis of $TL_n$.
\end{lem}
\begin{proof}
Since each element in $\mathfrak{B}_n$ can be written as a product of $e_i$'s,
we can write each as a sum of products of elements in $\mathfrak{D}_n$ by Lemma \ref{9}.
Moreover, a product of elements in $\mathfrak{D}_n$ can be written as a linear combination of elements in $\mathfrak{D}_n$ by Proposition \ref{4}.
So we can write elements in $\mathfrak{B}_n$ as sums of elements in $\mathfrak{D}_n$.
As $\mathfrak{B}_n$ is a basis for $TL_n$, the lemma holds.
\end{proof}

\section{Relation between $\mathfrak{B}_n$ and $\mathfrak{D}_n$}
\label{relation}
In this section, we will give a new system to denote the basis $\mathfrak{B}_n$.
We draw a diagram similar to elements in $\mathfrak{D}_n$ as in Figure \ref{f9},
except we do not put idempotents on strings and we put an empty circle at each black triple point.
If $a_i=a_{i+1}+1$, we put Figure \ref{f10} in the corresponding circle.
If $a_i=a_{i+1}-1$, then we put Figure \ref{f11} in the corresponding circle.
After filling all circles, we get a non-crossing diagram in $TL_n$ for each sequence $(a_1,...,a_{2n-1})$,
which satisfies the conditions in Definition \ref{11}.
Those elements belong to $\mathfrak{B}_n$.
Now, we give a total order on the set $\mathfrak{A}_n$ as follows:
$(a_1,...,a_{2n-1})<(b_1,...,b_{2n-1})$ if there is a $j$ such that $a_i=b_i$ for all $i<j$ and $a_j<b_j$.
This order on $\mathfrak{A}_n$ induces an order on $\{B_{a_1,...,a_{2n-1}}\}$ and $\mathfrak{D}_n$ naturally.
In this order, we will show that the representing matrix of $\{B_{a_1,...,a_{2n-1}}\}$ with respect to basis $\mathfrak{D}_n$ is upper triangular having $1$'s on the diagonal.

\begin{figure}[h]
\centering
\includegraphics[width=1in,height=1.2in]{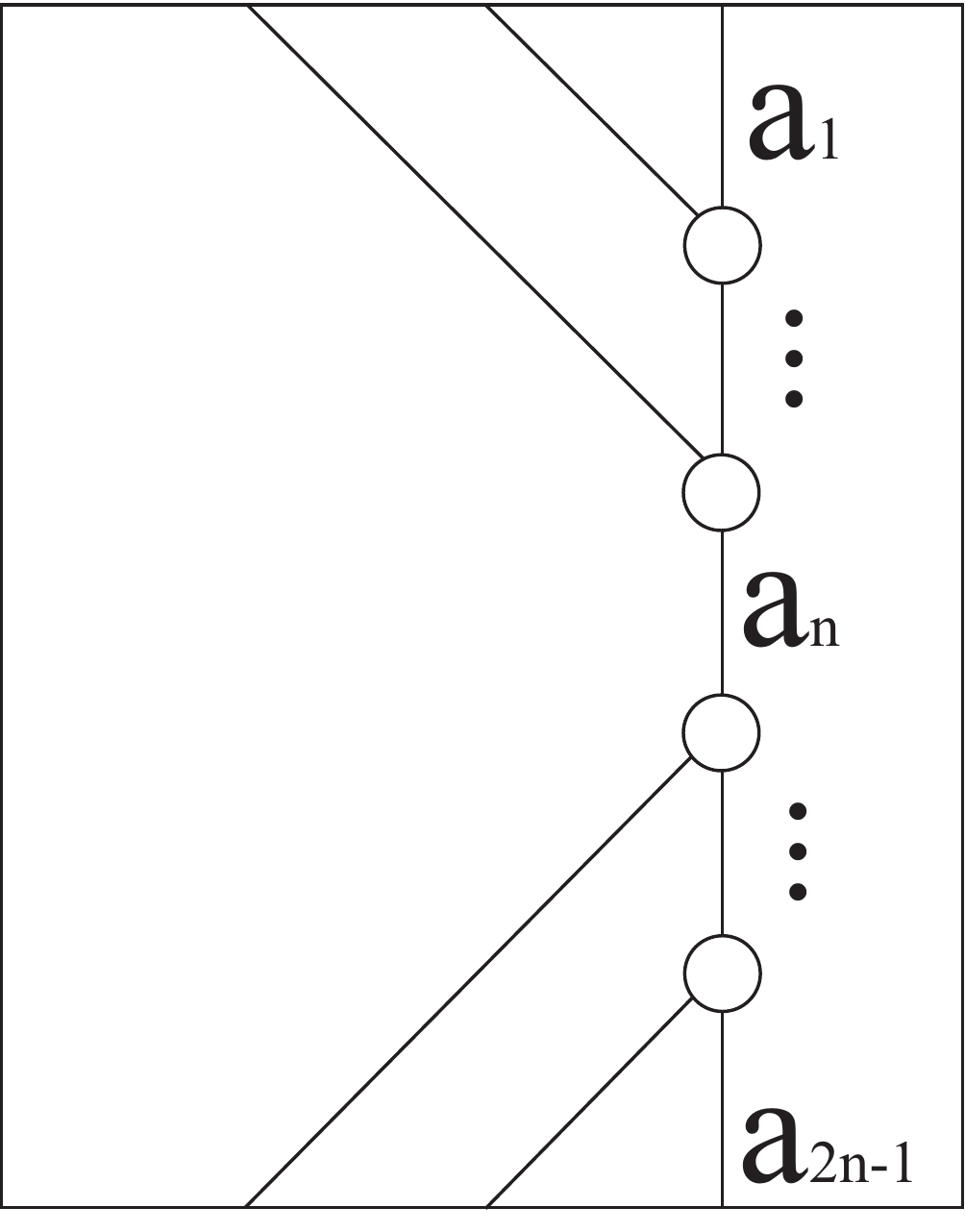}
\caption{An example of $B_{a_1,...,a_{2n-1}}$.}
\label{f9}
\end{figure}

\begin{figure}[h]
\begin{minipage}[h]{0.5\columnwidth}
\centering
\includegraphics[width=1in,height=1in]{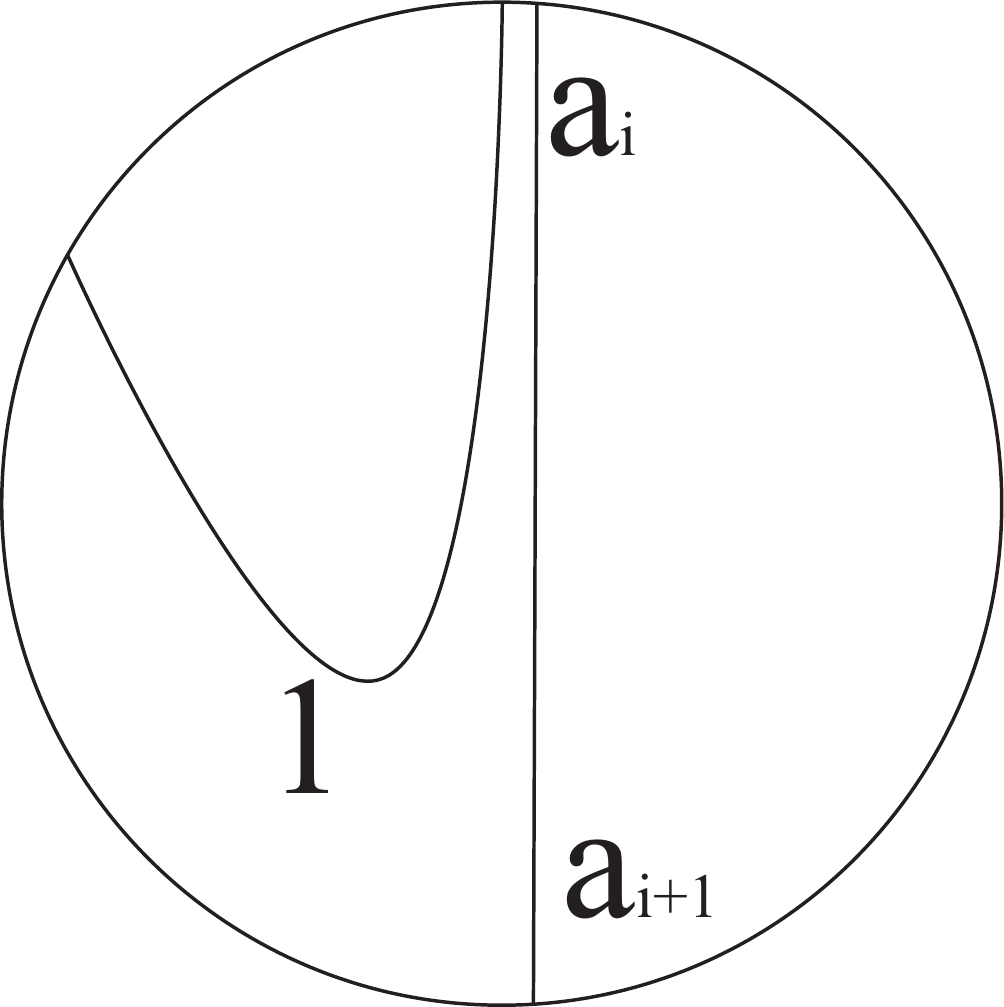}
\caption{$a_i=a_{i+1}+1$}
\label{f10}
\end{minipage}
\begin{minipage}[h]{0.5\columnwidth}
\centering
\includegraphics[width=1in,height=1in]{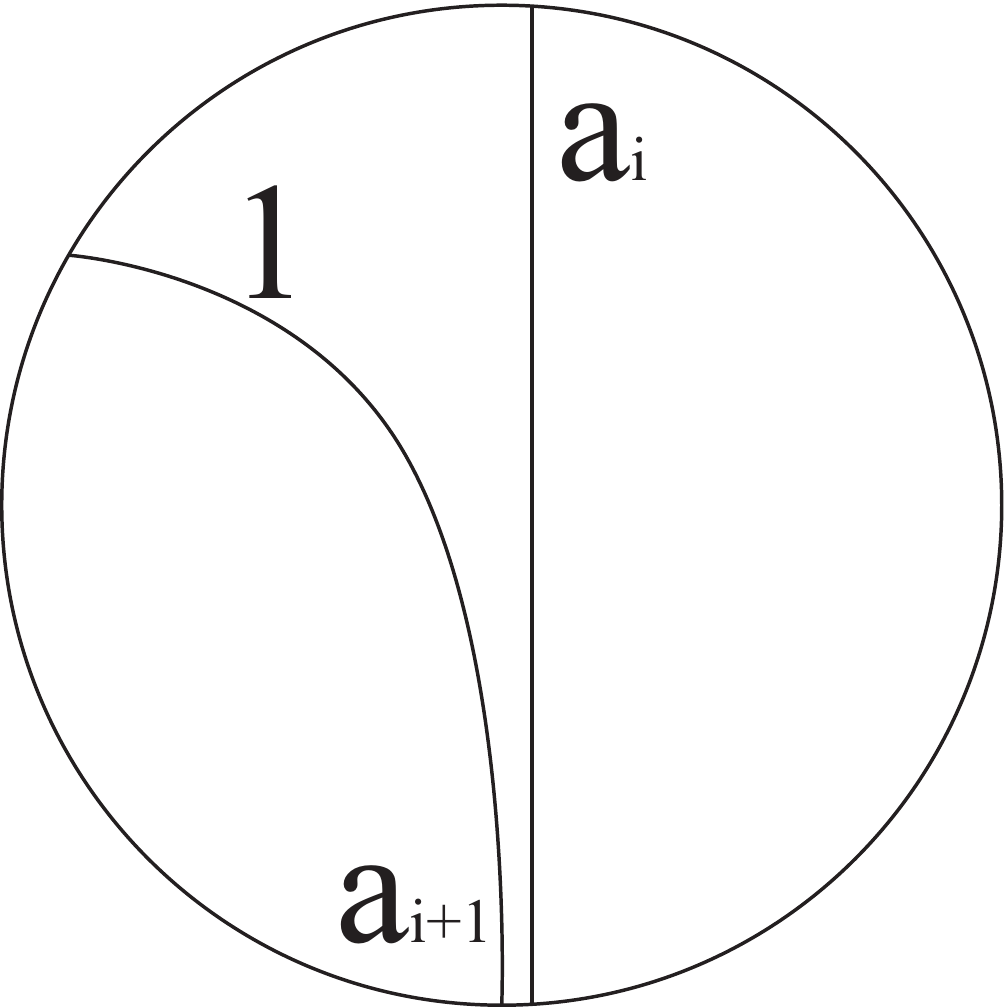}
\caption{$a_i=a_{i+1}-1$}
\label{f11}
\end{minipage}
\end{figure}

\begin{lem}\label{12}
$\langle B_{a_1,a_2,...,a_{2n-1}},D_{b_1,b_2,...,b_{2n-1}} \rangle=0$
if $(a_1,a_2,...,a_{2n-1})<(b_1,b_2,...,b_{2n-1})$.
\end{lem}
\begin{proof}
Since $(a_1,a_2,...,a_{2n-1})<(b_1,b_2,...,b_{2n-1})$, there is a $j$ such that $a_j<b_j$.
If we pair $B_{a_1,a_2,...,a_{2n-1}}$ and $D_{b_1,b_2,...,b_{2n-1}}$ together,
we can find a circle passing through them at $a_j$ and $b_j$.
We cut the pairing along this circle to get an element as in Figure \ref{f12}.
By the properties of idempotents, it is easy to see that this element is 0 in $\mathcal{S}(D^2)$ with $a_j$ and $b_j$ on the boundary.
So we get the result.
\end{proof}

\begin{figure}[h]
\centering
\includegraphics[width=1in,height=1in]{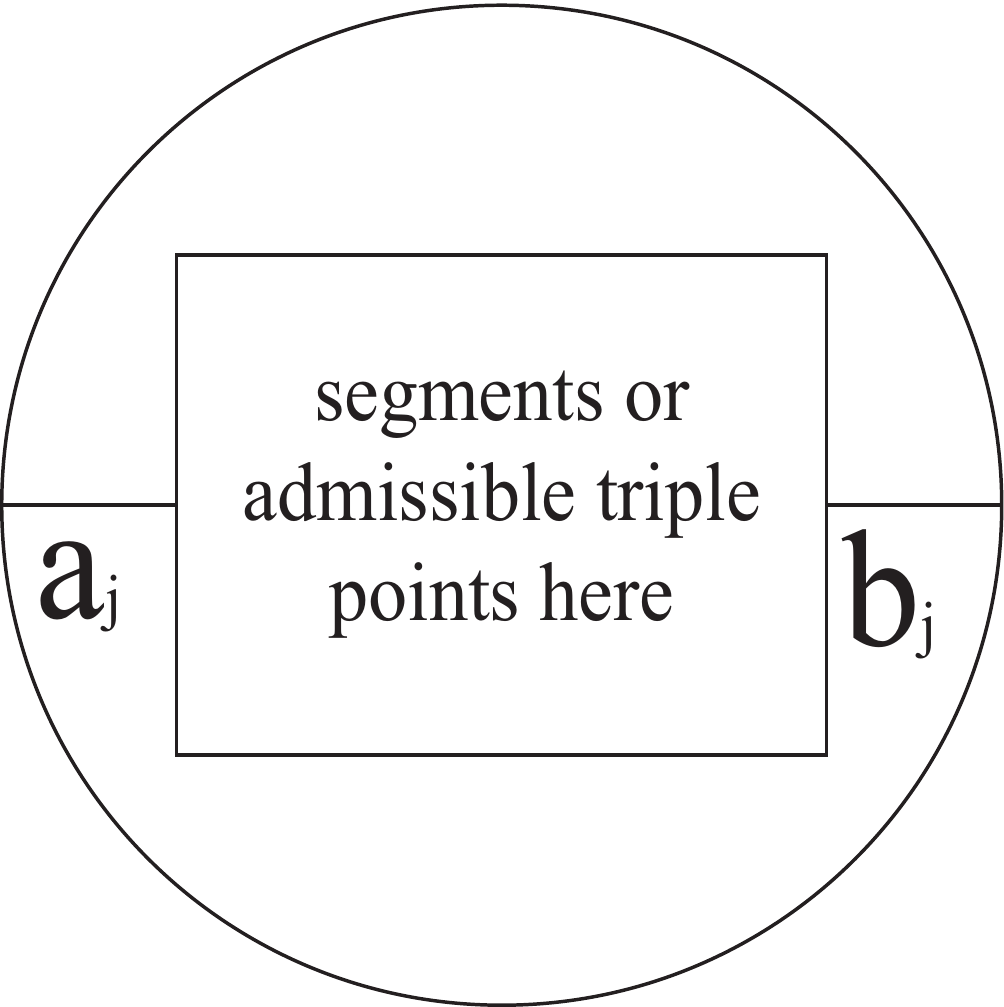}
\caption{We have the idempotent at $b_j$ and no idempotent at $a_j$.}
\label{f12}
\end{figure}
Before we go on, we introduce a lemma and a corollary.

\begin{lem}\label{15}
$\Theta(n,n+1,1)=\Delta_{n+1}$, and $\Theta(n,n-1,1)=\Delta_n$.
\end{lem}
\begin{proof}
	\begin{eqnarray}
		&&\Theta(n,n+1,1)\notag\\
		&=&\begin{minipage}{1in}\includegraphics[width=1in]{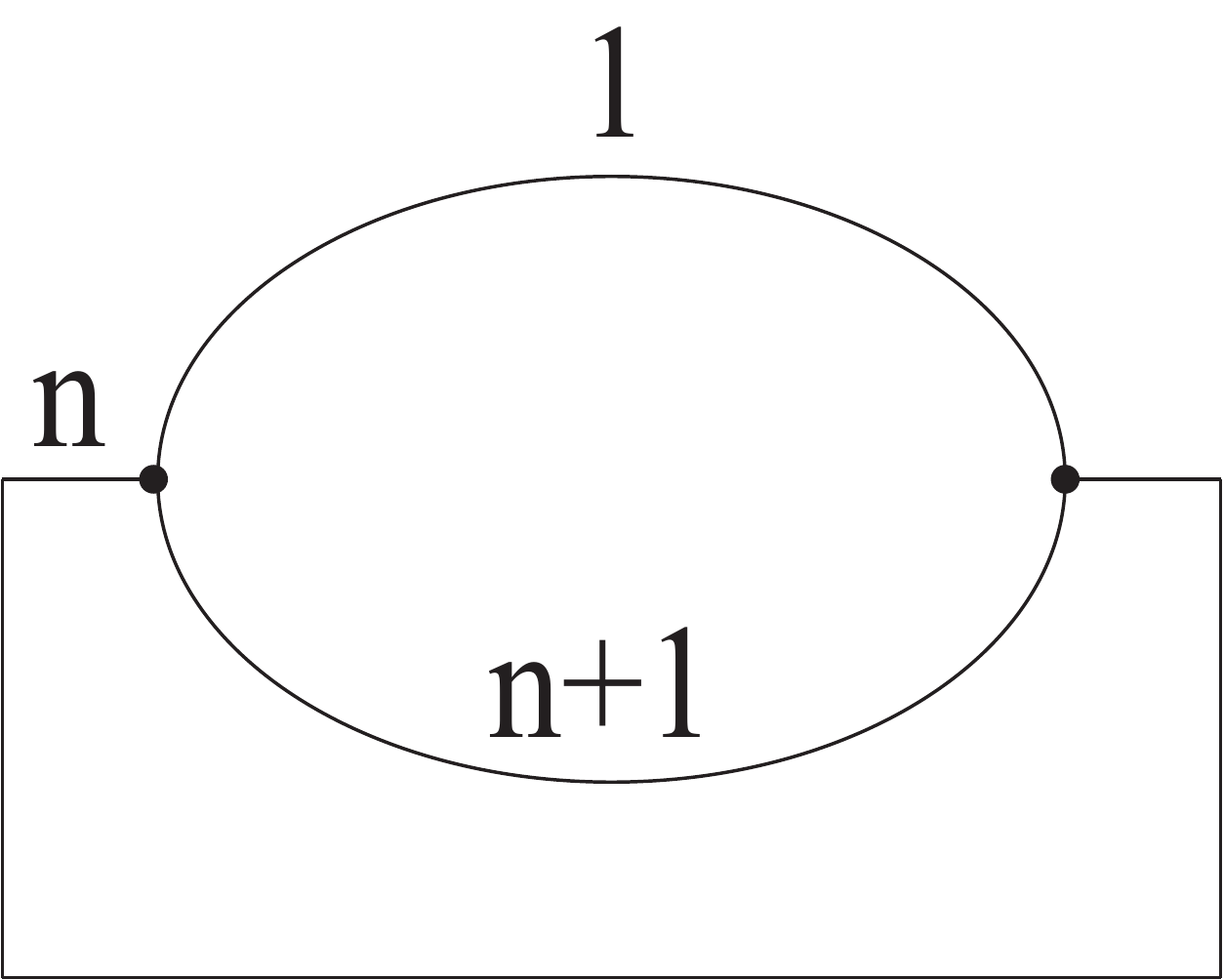}\end{minipage}=\begin{minipage}{1in}\includegraphics[width=1in]{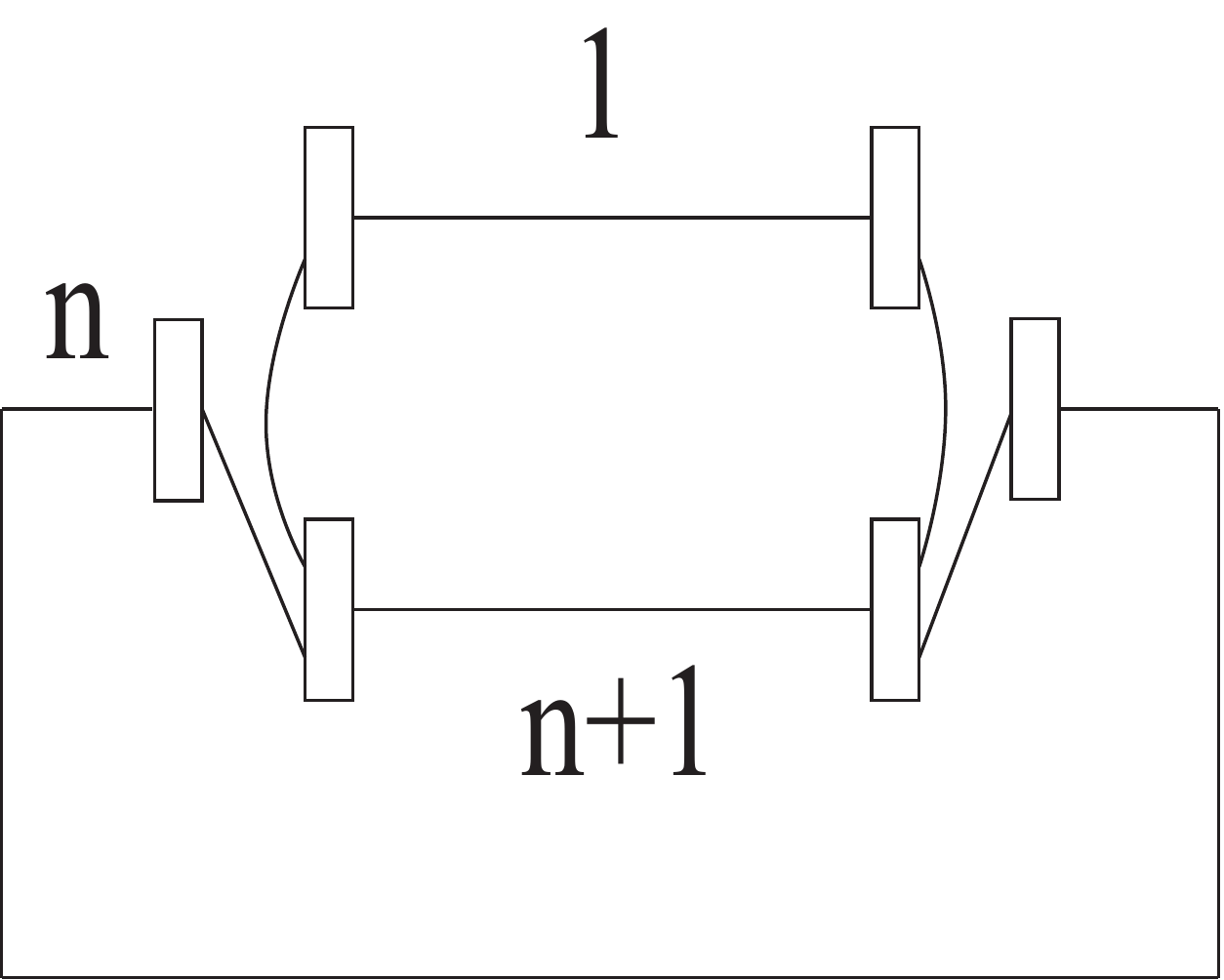}\end{minipage}\notag\\
		&=&\begin{minipage}{1in}\includegraphics[width=1in]{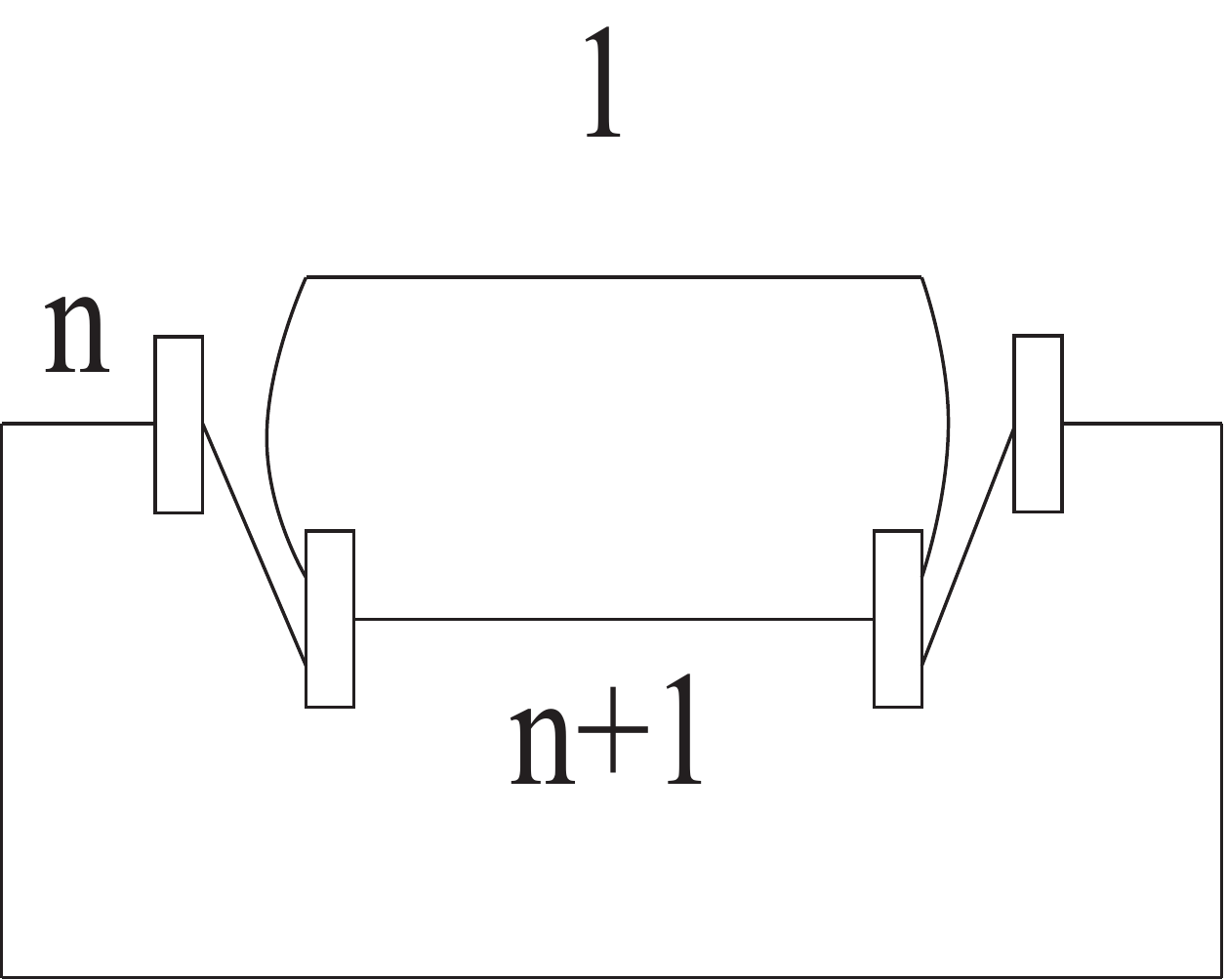}\end{minipage}=\begin{minipage}{1in}\includegraphics[width=1in]{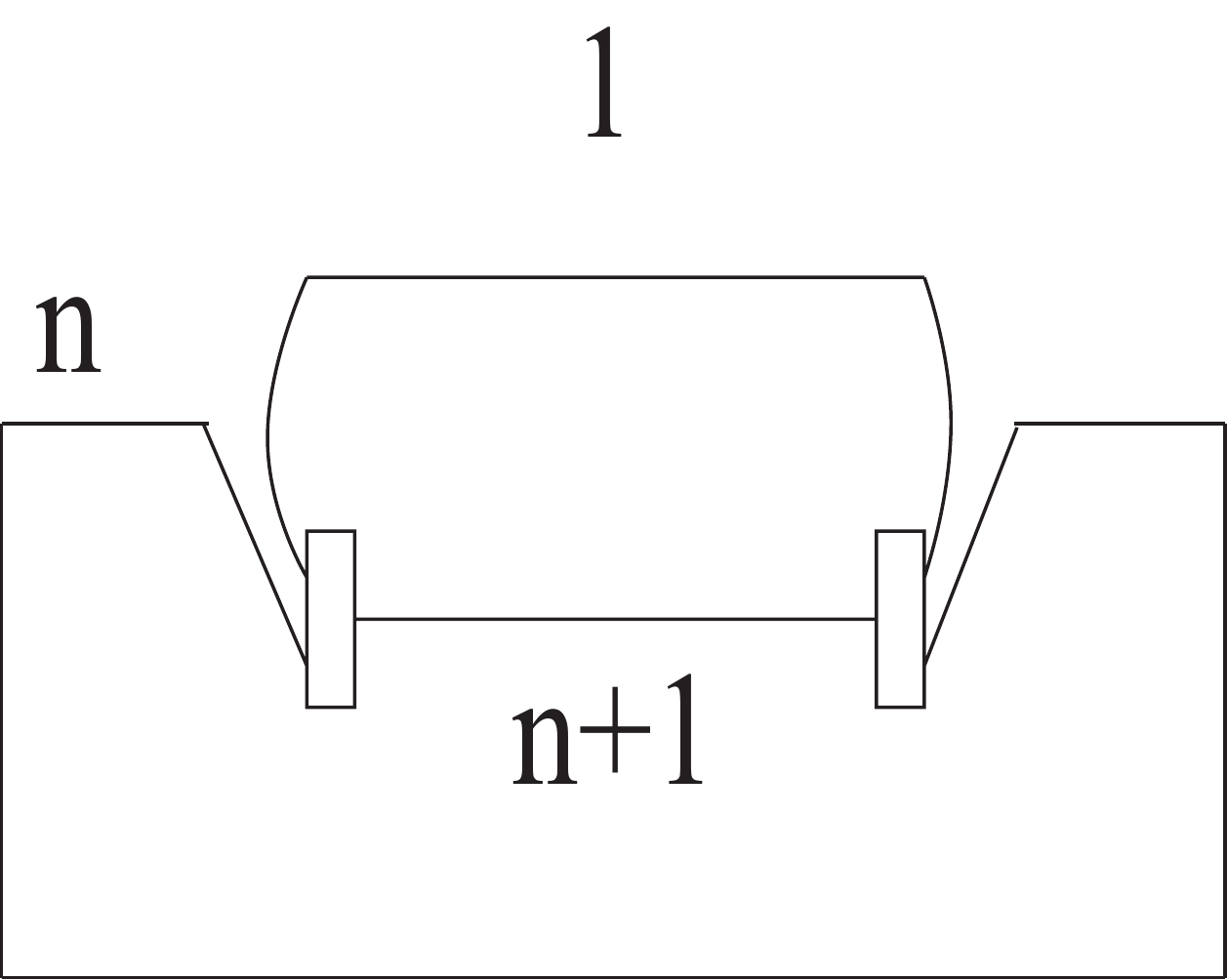}\end{minipage}\notag\\
		&=&\Delta_{n+1}\notag
	\end{eqnarray}
Similarly, it is easy to see that $\Theta(n,n-1,1)=\Delta_n$.
\end{proof}

\begin{cor}\label{10}
    \[
        \Gamma(b,a) =
        \begin{cases}

                1                             &\text{if $a=b+1$,} \\
                \frac{\Delta_{a+1}}{\Delta_a} &\text{if $a=b-1$.}

        \end{cases}
    \]
\end{cor}
\begin{proof}
This follows easily from Lemma \ref{15}.
\end{proof}
Now, we can prove the following:

\begin{prop}\label{13}
$\langle B_{a_1,a_2,...,a_{2n-1}},D_{a_1,a_2,...,a_{2n-1}} \rangle\ =\ \langle D_{a_1,a_2,...,a_{2n-1}},D_{a_1,a_2,...,a_{2n-1}} \rangle$ for all ${(a_1,a_2,...,a_{2n-1})}$ in $\mathfrak{A}_n$.
\end{prop}
\begin{proof}
We prove this by induction on $n$.
Suppose $n=2$. Then it is easy to check that
\begin{align*}
\langle B_{1,0,1},D_{1,0,1} \rangle\ =\ \langle D_{1,0,1},D_{1,0,1} \rangle, \langle B_{1,2,1},D_{1,2,1} \rangle\ =\ \langle D_{1,2,1},D_{1,2,1} \rangle.
\end{align*}
Assume that the result is true for $n<k$.
We will prove it is true for $n=k$. For $(a_1,...,a_{2n-1})=(1,2,...,k,...,2,1)$, we have
\begin{align*}
\langle B_{1,2,...,k,...,2,1},D_{1,2,...,k,...,2,1} \rangle\ =\ \langle D_{1,2,...,k,...,2,1},D_{1,2,...,k,...,2,1} \rangle
\end{align*}
by direct computation. 
For $(a_1,...,a_{2n-1})\neq (1,2,...,k,...,2,1)$, we choose $i>k>j$ such that $a_{i-1}=a_{i+1}$, $a_{j-1}=a_{j+1}$ and $a_{i}=a_{i-1}+1$, $a_{j}=a_{j-1}+1$.
Then by Proposition \ref{4}, we have
\begin{eqnarray*}
&&\langle B_{a_1,a_2,...,a_{2n-1}},D_{a_1,a_2,...,a_{2n-1}} \rangle \\
&=&\Gamma(a_i,a_{i+1})\Gamma(a_j,a_{j+1})\\
&&\times\langle B_{a_1,...,\hat{a}_i,\hat{a}_{i+1},...\hat{a}_j,\hat{a}_{j+1},...,a_{2n-1}},D_{a_1,...,\hat{a}_i,\hat{a}_{i+1},...,\hat{a}_j,\hat{a}_{j+1},...,a_{2n-1}} \rangle
\end{eqnarray*}
where the hat on $a_i$ means that it is removed.\\
It is easy to see that 
\begin{align*}
B_{a_1,...,\hat{a}_i,\hat{a}_{i+1},...\hat{a}_j,\hat{a}_{j+1},...,a_{2n-1}}\in\mathfrak{B}_{n-2},
D_{a_1,...,\hat{a}_i,\hat{a}_{i+1},...,\hat{a}_j,\hat{a}_{j+1},...,a_{2n-1}}\in\mathfrak{D}_{n-2}.
\end{align*}
By Lemma \ref{7} and Corollary \ref{10},
\begin{eqnarray*}
&&\langle D_{a_1,a_2,...,a_{2n-1}},D_{a_1,a_2,...,a_{2n-1}} \rangle\\
&=&\Gamma(a_i,a_{i+1})\Gamma(a_j,a_{j+1})\\
&&\times\langle D_{a_1,...,\hat{a}_i,\hat{a}_{i+1},...\hat{a}_j,\hat{a}_{j+1},...,a_{2n-1}},D_{a_1,...,\hat{a}_i,\hat{a}_{i+1},...,\hat{a}_j,\hat{a}_{j+1},...,a_{2n-1}} \rangle.
\end{eqnarray*}
By induction,
\begin{eqnarray*}
&&\langle B_{a_1,...,\hat{a}_i,\hat{a}_{i+1},...\hat{a}_j,\hat{a}_{j+1},...,a_{2n-1}},D_{a_1,...,\hat{a}_i,\hat{a}_{i+1},...,\hat{a}_j,\hat{a}_{j+1},...,a_{2n-1}} \rangle \nonumber \\
&=&\langle D_{a_1,...,\hat{a}_i,\hat{a}_{i+1},...\hat{a}_j,\hat{a}_{j+1},...,a_{2n-1}},D_{a_1,...,\hat{a}_i,\hat{a}_{i+1},...,\hat{a}_j,\hat{a}_{j+1},...,a_{2n-1}} \rangle.
\end{eqnarray*}
Hence,
\begin{align*}
\langle B_{a_1,a_2,...,a_{2n-1}},D_{a_1,a_2,...,a_{2n-1}} \rangle = \langle D_{a_1,a_2,...,a_{2n-1}},D_{a_1,a_2,...,a_{2n-1}} \rangle.
\end{align*}
\end{proof}

\begin{prop}\label{14}
\[
	\begin{pmatrix}
	B_{1,2,...,n,...,1}\\
	\vdots\\
	B_{1,0,1,0,...,0,1}
	\end{pmatrix}
	=
	\begin{pmatrix}
  1 & * & \cdots &* \\
  0 & 1 & \cdots &* \\
  \vdots  & \vdots  & \ddots & \vdots  \\
  0 & 0 & \cdots & 1
  \end{pmatrix}
 	\begin{pmatrix}
	D_{1,2,...,n,...,1}\\
	\vdots\\
	D_{1,0,1,0,...,0,1}
	\end{pmatrix}.
\]
\end{prop}
\begin{proof}
This follows easily from Lemma \ref{12} and Proposition \ref{13}.
\end{proof}

\begin{cor}
\label{16}
$\{B_{a_1,...,a_{2n-1}}\}=\mathfrak{B}_n$.
\end{cor}
\begin{proof}
By Proposition \ref{14}, we can see that $B_{a_1,...,a_{2n-1}}\neq B_{b_1,...,b_{2n-1}}$ if $(a_1,...,a_{2n-1})\neq(b_1,...,b_{2n-1})$.
Moreover, $\card\{B_{a_1,...,a_{2n-1}}\}=\card\mathfrak{D}_n=\card\mathfrak{B}_n$ and $\{B_{a_1,...,a_{2n-1}}\}\subset\mathfrak{B}_n$.
So we have $\{B_{a_1,...,a_{2n-1}}\}=\mathfrak{B}_n$.
\end{proof}

By Corollary \ref{16} and linear algebra,
we can see that the $\det(G_n)$ we get by using the basis $\mathfrak{B}_n$ is the same as $\det(G_n)$ we get by using the basis $\mathfrak{D}_n$.

\section{Lattice path}
\label{latticepath}

\begin{de}
A lattice path in the plane is a path from $(0,0)$ to $(a,b)$ with northeast and
southeast unit steps, where $a,b\in\mathbb{Z}$. A Dyck path is a lattice path that
never goes below the $x-$axis. We denote the set of all Dyck path from $(0,0)$ to
$(a,b)$ by $\mathcal{D}_{(a,b)}$.
\end{de}

\begin{figure}[h]
\centering
   \includegraphics[width=1in,height=1.2in]{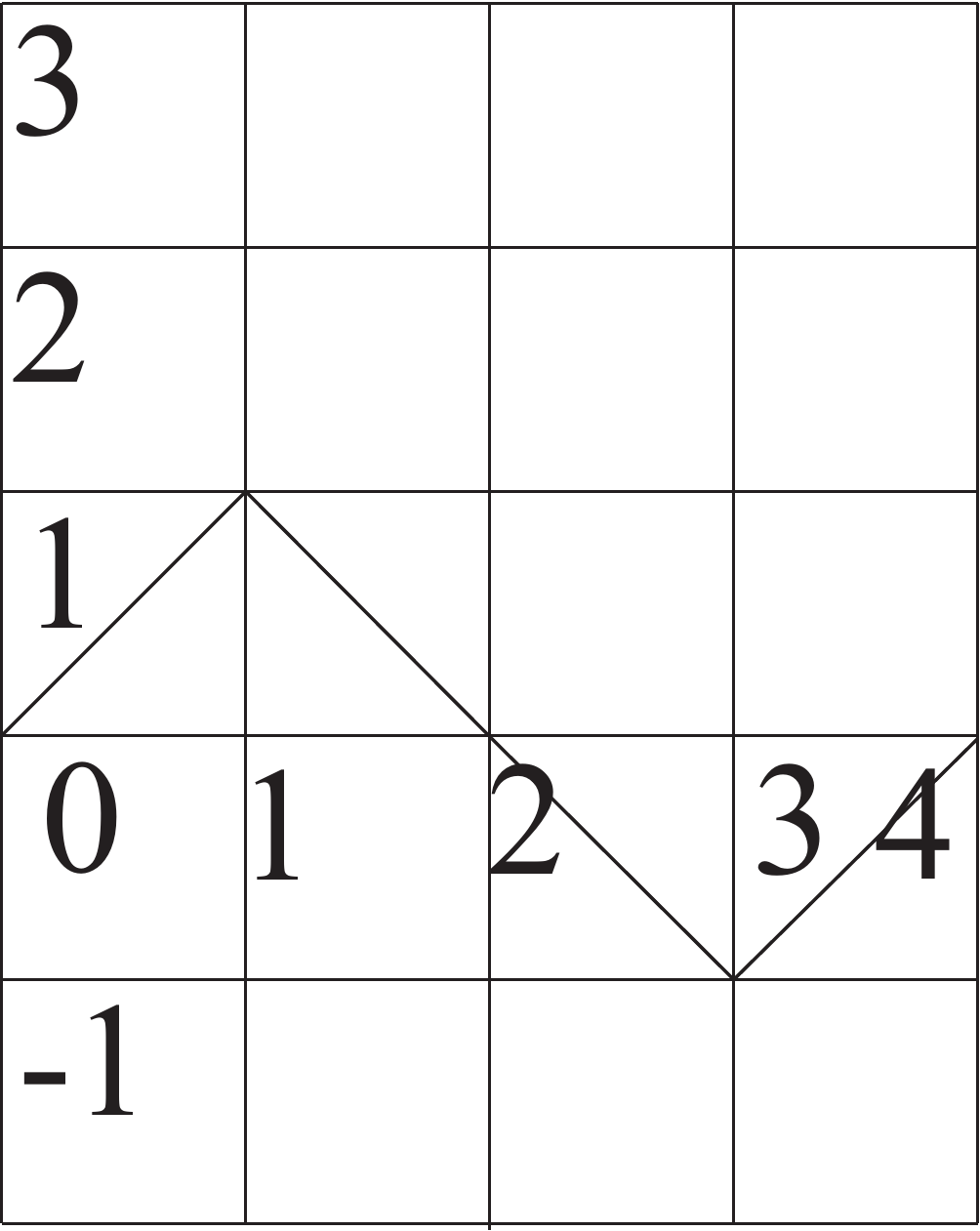}\hskip 0.5in
   \includegraphics[width=1in,height=1.2in]{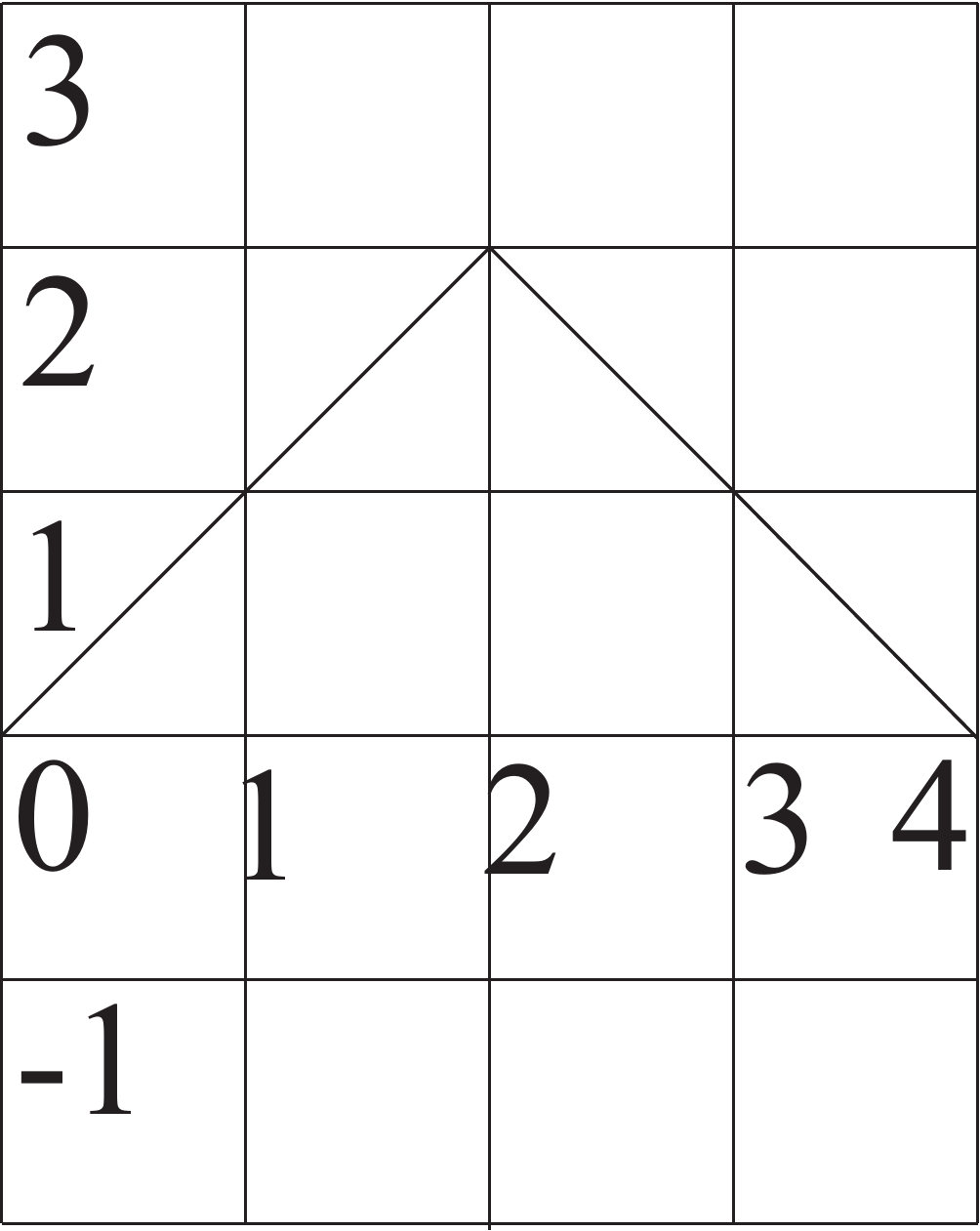}
   \caption{On the left is a lattice path. On the right is a Dyck path.}
\end{figure}

\begin{rem}
There is a natural bijection $f$ from $\mathfrak{D}_n$ to $\mathcal{D}_{(2n,0)}$, the set of all Dyck paths from $(0,0)$ to $(2n,0)$ as follows:\\
For each $D_{a_1,...,a_{2n-1}}\in\mathfrak{D}_n$, we construct a path from $(0,0)$ to $(2n,0)$ with step $(i,a_i)$
for all $1\leq i\leq 2n-1$. Since $a_i$ satisfies 	
\begin{enumerate}
	\item $a_1=a_{2n-1}=1$;
	\item $a_i\in \mathbb{N}$ for all $i$;	
	\item $\|a_i-a_{i-1}\|=1$ for all $i$.
\end{enumerate}
We can see that this is a Dyck path.
\end{rem}

\begin{rem}
The reflection principle \cite[page 22]{C} says that the number of all Dyck paths from $(0,0)$ to $(2n,0)$ is the Catalan number $C_n=\frac{1}{n+1}\binom{2n}{n}$.
Hence, we recover the well-known result that the dimension of $TL_n$ is $C_n$.
\end{rem}

\section{Proof of the Main Theorem}
\label{proof}
Now we can start our proof of the main theorem. 
By Lemma \ref{8}, we know that $\{D_{a_1,...,a_{2n-1}}\}$ is an orthogonal basis with respect to the bilinear form.
Thus the matrix of $G_n$ is a diagonal matrix under this basis.
We have
\begin{align*}
\det(G_n)=\prod_{(a_1,...,a_{2n-1})}\langle D_{a_1,...,a_{2n-1}},D_{a_1,...,a_{2n-1}}\rangle.
\end{align*}
Then by Lemma \ref{7},
we have
\begin{align*}
\det(G_n)\nonumber=\prod_{(a_1,...,a_{2n-1})}\Gamma(a_1,a_2)\Gamma(a_2,a_3)\dots\Gamma(a_{2n-2},a_{2n-1})\Delta_1.
\end{align*}
Using Lemma \ref{10}, we can simplify $\det(G_n)$ as follows:\\
Consider the tuple $(D,i)$ such that $D$ is an element of $\mathfrak{D}_n$ and $a_i=k$ in $D$.
If $a_{i+1}=k+1$, then $\Gamma(a_i,a_{i+1})$ is 1 by Lemma \ref{10}.
So $(D,i)$ will contribute 1 to $\det(G_n)$.
If $a_{i+1}=k-1$, then $\Gamma(a_i,a_{i+1})=\frac{\Delta_{k+1}}{\Delta_k}$.
So $(D,i)$ will contribute $\frac{\Delta_k}{\Delta_{k-1}}$ to $\det(G_n)$.
We denote  by $\mathfrak{S}_k$ the set of all tuple $\{(D,i)\}$ with $D\in\mathfrak{D}_n$ and $a_i=k,a_{i+1}=k-1$ in $D$.
Let $\alpha_k$ be the cardinality of $\mathfrak{S}_k$. Then
\begin{align*}
\det(G_n) = \Delta_1^{\card\mathfrak{D}_n}\prod_{k=1}^n(\frac{\Delta_k}{\Delta_{k-1}})^{\alpha_k}\ .
\end{align*}
Now, the theorem is reduced to calculate $\alpha_k$ for each $k$.

\begin{prop}
\[ \alpha_k =
\left(
\begin{array}{c}
2n \\
n-k
\end{array}
\right) -
\left(
\begin{array}{c}
2n \\
n-k-1
\end{array}
\right).
\]
\end{prop}
\begin{proof}
In Section \ref{latticepath}, we already had a 1-1 correspondence $f$ between $\mathfrak{D}_n$ (the new basis we constructed) 
and $\mathcal{D}_{(2n,0)}$ (Dyck paths from $(0,0)$ to $(2n,0)$).
With respect to this correspondence, each pair $(D,i)$ in $\mathfrak{S}_k$ is associated to a pair $(f(D),i)$ with $a_i=k,a_{i+1}=k-1$ in $f(D)$, 
that is the step from $(i,a_i=k)$ to $(i+1,a_{i+1}=k-1)$ in the path $f(D)$.
See Figure \ref{f4} for an example.
Denote by $\mathcal{S}_k$ the set of all pairs $(f(D),i)$, where $f(D)$ has a step from $k$ to $k-1$ at $i$.
Thus we have a 1-1 correspondence between $\mathfrak{S}_k$ and $\mathcal{S}_k$.
\begin{figure}[h]
\centering
   \includegraphics[width=1.2in,height=1.2in]{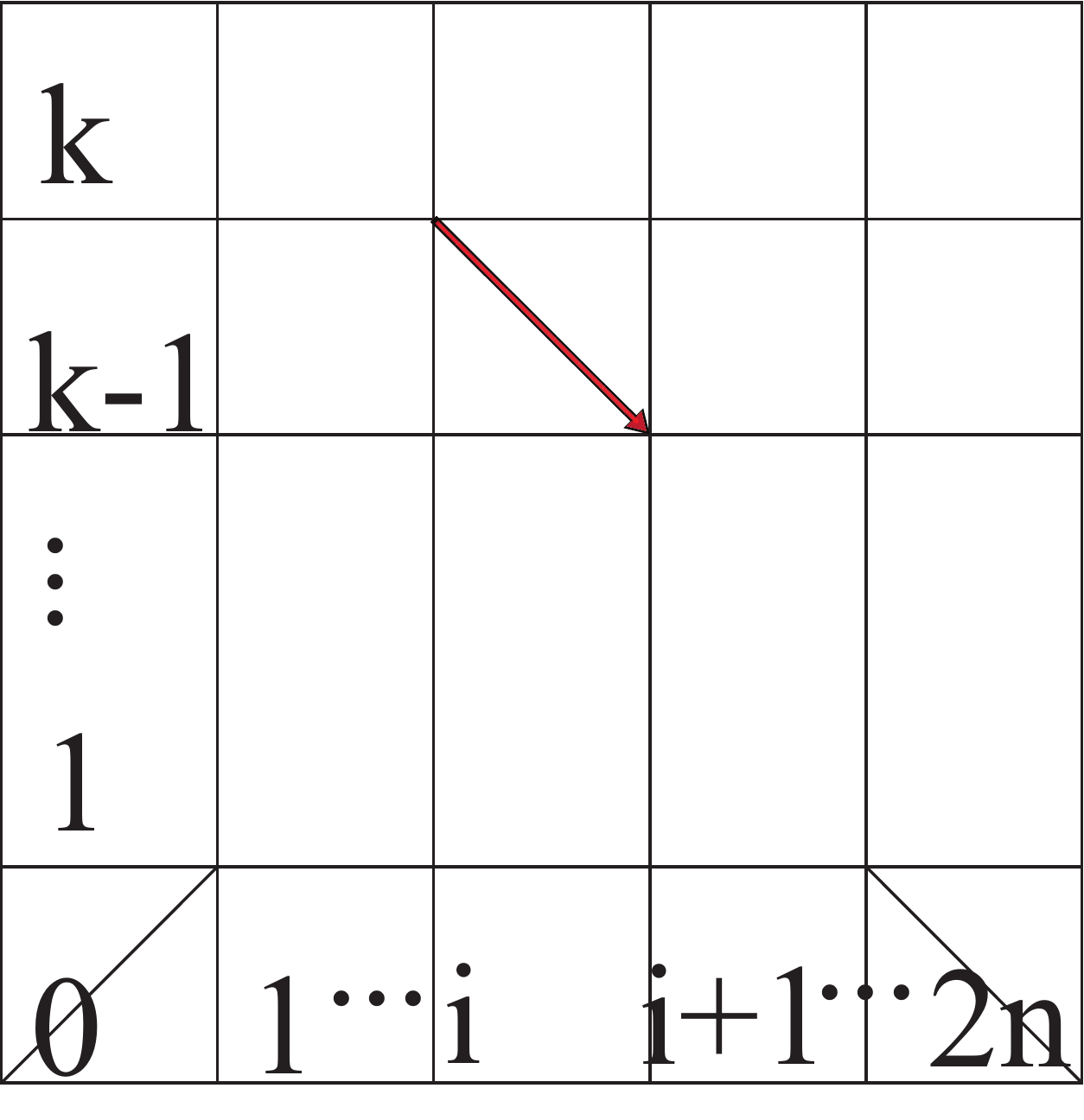}
   \caption{The bold step is a step going down from $(i,a_i=k)$ to $(i+1,a_{i+1}=k-1)$.}
   \label{f4}
\end{figure}
Di Francesco \cite[page 562]{F} set up a 1-1 correspondence from $\mathcal{S}_k$ to $\mathcal{D}_{(2n,2k)}$.
Then we have
\[ \alpha_k=\card{\mathcal{D}_{(2n,2k)}}=
\left(
\begin{array}{c}
2n \\
n-k
\end{array}
\right) -
\left(
\begin{array}{c}
2n \\
n-k-1
\end{array}
\right).
\]

For the convenience of reader, we now give Di Francesco's correspondence in our terminology.
For an element $\hat{P}\in\mathcal{D}_{(2n,2k)}$, it should intersect the horizontal line $y=k$ in a point $p=(i,a_i=k)$ with $a_{i+1}=k+1$ and $a_{i-1}=k-1$ at least once.
Let $p$ be the rightmost such intersection. Now we cut $\hat{P}$ at the point $p$, 
reflect the right part of $\hat{P}$ with respect to $y$-axis, and shift it down by $k$ units.
Then we glue this part back to the left part. We get a Dyck path $P$ from $(0,0)$ to $(2n,0)$.
In the resulting path $P$, we then choose the smallest $i'\geq i$, such that $a'_i=k$ and $a_{i'+1}=k-1$.
We associate the path $\hat{P}$ to the pair $(P,i')$.
Therefore, we construct a map $\phi:\mathcal{D}_{(2n,2k)}\rightarrow\mathcal{S}_k$.\\
Conversely, for a pair $(P,i)$, where $P\in\mathcal{D}_{(2n,0)}$.
We choose the largest $i'\leq i$ with $a_{i'}=k$ and $a_{i'-1}=k-1$.
We cut the path $P$ at $i'$, reflect the right part with respect to the $y$-axis, shift up by $n$ units and glue it back.
Thus we construct a path $\hat{P}$ in $\mathcal{D}_{(2n,2k)}$.
We associate the tuple $(P,i)$ to the path $\hat{P}$.
Therefore, we construct a map $\varphi:\mathcal{S}_k\rightarrow\mathcal{D}_{(2n,2k)}$.\\
It is easy to see that $\phi\varphi=id$ and $\varphi\phi=id$.
Therefore, we have constructed a 1-1 correspondence between $\mathcal{S}_k$ and $\mathcal{D}_{(2n,2k)}$.

By the reflection principle, we have
\[ \card\mathcal{D}_{(2n,2k)} =
\left(
\begin{array}{c}
2n \\
n-k
\end{array}
\right) -
\left(
\begin{array}{c}
2n \\
n-k-1
\end{array}
\right).
\]
Therefore, we have

\[ \alpha_k =
\left(
\begin{array}{c}
2n \\
n-k
\end{array}
\right) -
\left(
\begin{array}{c}
2n \\
n-k-1
\end{array}
\right).
\]
\end{proof}

\section{Relation between $\mathfrak{D}_n$ and Di Francesco's second basis.}
In this section, we extend our ring $\Lambda$ to the complex numbers $\mathbb{C}$ and let $A$ be any non-zero complex number which is not a root of unity.

\begin{de}
\begin{equation}
ND_{a_1,...,a_{2n-1}}=\frac{D_{a_1,...,a_{2n-1}}}{<D_{a_1,...,a_{2n-1}},D_{a_1,...,a_{2n-1}}>^{\frac{1}{2}}}.
\notag
\end{equation}
We call $ND_{a_1,...,a_{2n-1}}$ the normalization of $D_{a_1,...,a_{2n-1}}$, 
and denote the normalized basis by $\mathfrak{ND}_n$.
\end{de}

\begin{thm}
$\mathfrak{ND}_n$ is the same basis as Di Francesco's second basis in \cite{F}.
\end{thm}

\begin{proof}
Di Francesco defined his orthonormal basis by a recursive equation \cite[equation 3.19, Page 555]{F}.
So we just need to show that $\mathfrak{ND}_n$ satisfies the recursive equation and the initial condition.

Let $D_{a_1,...,a_{i-1},a_i,a_{i+1},...,a_{2n-1}}$ and $D_{a_1,...,a_{i-1},a'_i,a_{i+1},...,a_{2n-1}}$ be two elements in $\mathfrak{D}_n$ such that $a_i=a_{i-1}-1=a_{i+1}-1$ and $a'_i=a_{i}+2$,
that means they are equal everywhere except at $i$th arc.
Then using the recursive formula for Jones-Wenzl idempotents at $i$th arc,
we have 
\begin{equation}
D_{a_1,...,a_{i-1},a_{i'},a_{i+1},...,a_{2n-1}}=D_{a_1,...,a_{i-1},\hat{a}_i,a_{i+1},...,a_{2n-1}}-\frac{\Delta_{a_i}}{\Delta_{a_i+1}}D_{a_1,...,a_{i-1},a_i,a_{i+1},...,a_{2n-1}}
\notag
\end{equation}
where $D_{a_1,...,a_{i-1},\hat{a}_i,a_{i+1},...,a_{2n-1}}$ is as in Figure \ref{f14}.
\begin{figure}[h]
\centering
\includegraphics[width=1in,height=1.2in]{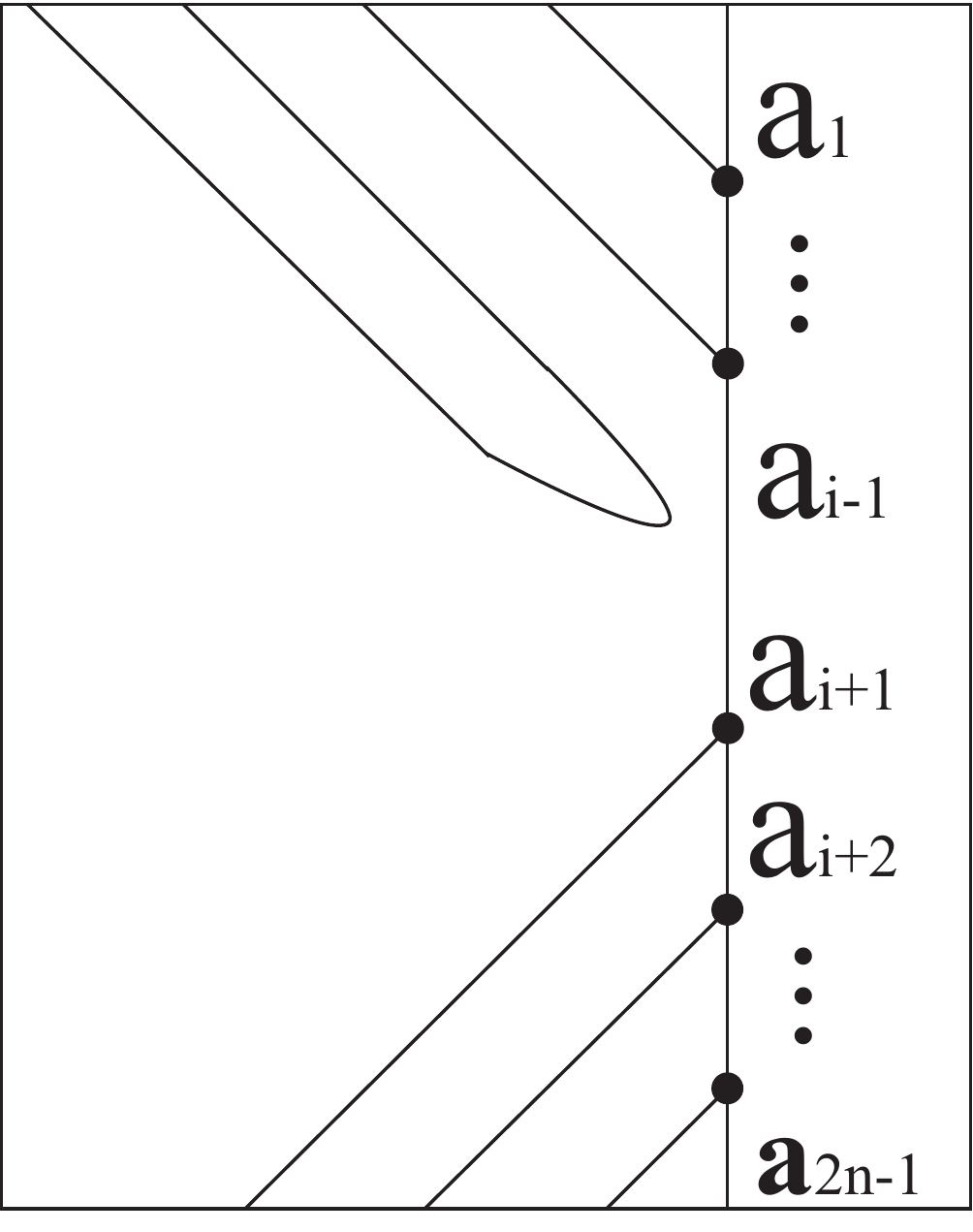}
\caption{$a_i$ disappears.}
\label{f14}
\end{figure}
Since $a_{i-1}=a_{i+1}$, this is well defined.
It is easy to see that $D_{a_1,...,a_{i-1},\hat{a}_i,a_{i+1},...,a_{2n-1}}=e_iD_{a_1,...,a_{i-1},a_{i},a_{i+1},...,a_{2n-1}}$,
where $e_i$ acts on $D_{a_1,...,a_{i-1},a_{i},a_{i+1},...,a_{2n-1}}$ as in \cite{F}.
We divide the equation by the norm of $D_{a_1,...,a_{i-1},a_{i'},a_{i+1},...,a_{2n-1}}$ on both sides.
By Lemma \ref{7}, we have 
\begin{eqnarray}
&&ND_{a_1,...,a_{i-1},a'_i,a_{i+1},...,a_{2n-1}}\notag\\
&=&\frac{(e_iD_{a_1,...,a_{i-1},a_i,a_{i+1},...,a_{2n-1}}-\frac{\Delta_{a_i}}{\Delta_{a_i+1}}D_{a_1,...,a_{i-1},a_i,a_{i+1},...,a_{2n-1}})}
{(\Gamma(a_1,a_2)\dots\Gamma(a_{i-1},a'_i)\Gamma(a'_i,a_{i+1})\dots\Gamma(a_{2n-2},a_{2n-1})\Delta_{a_{2n-1}})^{\frac{1}{2}}}\notag\\
&=&(e_i-\frac{\Delta_{a_i}}{\Delta_{a_i+1}})ND_{a_1,...,a_{i-1},a_i,a_{i+1},...,a_{2n-1}}\times\notag\\
&&\frac{(\Gamma(a_1,a_2)\dots\Gamma(a_{i-1},a_{i})\Gamma(a_{i},a_{i+1})\dots\Gamma(a_{2n-2},a_{2n-1})\Delta_{a_{2n-1}})^{\frac{1}{2}}}{(\Gamma(a_1,a_2)\dots\Gamma(a_{i-1},a'_i)\Gamma(a'_i,a_{i+1})\dots\Gamma(a_{2n-2},a_{2n-1})\Delta_{a_{2n-1}})^{\frac{1}{2}}}\notag\\
&=&(e_i-\frac{\Delta_{a_i}}{\Delta_{a_i+1}})ND_{a_1,...,a_{i-1},a_i,a_{i+1},...,a_{2n-1}}\frac{(\Gamma(a_{i-1},a_{i})\Gamma(a_{i},a_{i+1}))^{\frac{1}{2}}}{(\Gamma(a_{i-1},a'_i)\Gamma(a'_i,a_{i+1}))^{\frac{1}{2}}}\notag\\
&=&(e_i-\frac{\Delta_{a_i}}{\Delta_{a_i+1}})ND_{a_1,...,a_{i-1},a_i,a_{i+1},...,a_{2n-1}}\frac{(\Gamma(a_{i-1},a_{i}))^{\frac{1}{2}}}{(\Gamma(a'_i,a_{i+1}))^{\frac{1}{2}}}\ .
\notag
\end{eqnarray}
By definition,
\begin{equation} 
\Gamma(a_{i-1},a_{i})=\mu_{a_i}, \Gamma(a'_i,a_{i+1})=\mu_{a_{i+1}},
\notag
\end{equation}
where $\mu_i$ as in \cite{F}.
Thus, $\mathfrak{ND}_n$ satisfies the recursive equation.
Moreover, it is easy to see that 
\begin{equation}
u_n=ND_{1,0,1,0,...,0,1,0,1},
\notag
\end{equation}
where $u_n$ is as in \cite[equation 3.5,Page 551]{F}.
So $\mathfrak{ND}_n$ satisfies the initial condition.
\end{proof}

\section*{Acknowledgements}
The author thanks his advisor Professor Gilmer for helpful discussions, the referee for useful comments and Peizhe Shi and Meng Yu for help on \LaTeX{}. The author was partially supported by a research assistantship funded by NSF-DMS-0905736.

\end{document}